\documentclass{article}
\usepackage[a4paper,twoside=false,footskip=1cm]{geometry}
\usepackage{hyperref}
\usepackage{enumitem}
\usepackage[intlimits]{amsmath}
\usepackage{amssymb,amsthm}
\usepackage{enumitem}
\usepackage{manfnt}
\usepackage{cleveref}
\usepackage[utf8]{inputenc}
\usepackage{pifont}
\usepackage{adjustbox}
\usepackage{dsfont}

\hypersetup{
linktoc=page,
    unicode=false,          % non-Latin characters in Acrobat’s bookmarks
    pdftoolbar=true,        % show Acrobats toolbar?
    pdfmenubar=true,        % show Acrobats menu?
    pdffitwindow=false,     % window fit to page when opened
    pdfstartview={FitH},    % fits the width of the page to the window
    pdftitle={Peacocks nearby: approximation of sequences of measures},    % title
    pdfauthor={S. Gerhold, I.C. G\"ul\"um},     % author
    pdfsubject={Revised version of paper, September 2017},   % subject of the document
    pdfcreator={ICG},   % creator of the document
    pdfproducer={ICG}, % producer of the document
    pdfkeywords= {Convex Order} {Infinity Wasserstein distance} {Kellerer's Theorem} {L\'evy distance} 
		{Peacock} {Prokhorov distance} {Stop-Loss Distance} {Strassen's Theorem}, % list of keywords
    pdfnewwindow=true,      % links in new window
    colorlinks=true,       % false: boxed links; true: colored links
    linkcolor=blue,          % color of internal links
    citecolor=blue,        % color of links to bibliography
    filecolor=blue,      % color of file links
    urlcolor=blue           % color of external links
}

\usepackage[english]{babel}

\theoremstyle{plain}
\newtheorem{thm}{Theorem}[section]
\newtheorem{lem}[thm]{Lemma}
\newtheorem{prop}[thm]{Proposition}
\newtheorem{cor}[thm]{Corollary}

\theoremstyle{definition}
\newtheorem{definition}[thm]{Definition}

\theoremstyle{remark}
\newtheorem{exa}[thm]{Example}
\newtheorem{rem}[thm]{Remark}

\newtheorem{pro}[thm]{Problem}

\theoremstyle{plain}
\newtheorem*{thm*}{Theorem}
\newtheorem*{lem*}{Lemma}
\newtheorem*{prop*}{Proposition}
\newtheorem*{cor*}{Corollary}
\newtheorem*{ass*}{Assumptions}

\theoremstyle{definition}
\newtheorem*{definition*}{Definition}

\theoremstyle{remark}
\newtheorem*{rem*}{Remark}
\newtheorem*{prob*}{Problem}
\newtheorem*{exa*}{Example}

\newcommand{\field}[1]{\mathbb{#1}}
\newcommand{\setN}{\field{N}}
\newcommand{\setQ}{\field{Q}}
\newcommand{\setR}{\field{R}}

\newcommand{\setM}{\mathcal{M}}

\newcommand{\pp}{{\mathbb P}}

\newcommand{\leqc}{\leq_{\mathrm{c}}}

\newcommand{\conv}{\mathrm{conv}}
\DeclareMathOperator{\sgn}{sgn}

\setlist[enumerate,1]{label=$(\roman*)$, ref=$(\roman*)$}
\setlist[enumerate,2]{label=\alph*), ref=\theenumi\ \alph*}

\usepackage{csquotes}

\begin{document}

%\title{A variant of Strassen's theorem: Existence of martingales within a prescribed distance}
\title{Peacocks nearby: approximating sequences of measures}
\author{Stefan Gerhold \\ TU Wien \\ sgerhold@fam.tuwien.ac.at
  \and I. Cetin G\"ul\"um\thanks{We acknowledge financial support from FWF under grant P~24880,
  and thank Mathias Beiglb\"ock and the anonymous referees for helpful comments.}
   \\ TU Wien \\ ismail.cetin.gueluem@gmx.at
}
\date{\today}

\maketitle

\begin{abstract}
A peacock is a family of probability measures with finite mean that
increases in convex order.
It is a classical
result, in the discrete time case due to Strassen, that any peacock is the family
of one-dimensional marginals of a martingale.
We study the problem whether a given sequence of probability measures can be approximated by a peacock.
In our main results, the approximation quality is measured by the infinity Wasserstein distance.
Existence of a peacock within a prescribed distance is reduced to a countable collection
of rather explicit conditions.
This result has a financial application (developed in a separate paper),
as it allows to check European call option quotes for consistency.
The distance bound on the peacock than takes the role of a bound on the bid-ask spread
of the underlying.
We also solve the approximation problem for
the stop-loss distance, the L\'evy distance, and the Prokhorov distance.
\end{abstract}

\section{Introduction}
\setcounter{equation}{0}
\numberwithin{equation}{section}
A celebrated result, first proved by Strassen in 1965,\footnote{See Theorem 8 in~\cite{St65}.
(Another result from that paper, relative to the usual stochastic order instead of the convex
order,
is also sometimes referred to as Strassen's theorem; see~\cite{Li99}.)
}
states that, for a given sequence of probability measures $(\mu_n)_{n \in \setN}$, 
there exists a martingale $M=(M_n)_{n \in \setN}$ such that the law of $M_n$ is $\mu_n$ for all~$n$, 
if and only if all $\mu_n$ have finite mean and $(\mu_n)_{n \in \setN}$ is increasing in convex order
(see Definition~\ref{def:conv}).
Such sequences, and their continuous time counterparts,
are nowadays referred to as peacocks, a pun on the French
acronym PCOC, for ``Processus Croissant pour l’Ordre Convexe''~\cite{HiPrRoYo11}.
For further references on Strassen's theorem and its predecessors, see the appendix
of~\cite{DaHo07}, p.~380 of Dellacherie and Meyer~\cite{DeMe88}, and~\cite{Ar13}.
A constructive proof, and references to earlier constructive proofs,
are given in M\"uller and R\"uschendorf~\cite{MuRu01}.

The theorem gave rise to plenty of generalizations, one of the most
famous being Kellerer's theorem~\cite{Ke72,Ke73}.
It states that, for a peacock $(\mu_t)_{t \geq 0}$ with index set $\setR^+$, there is a
\emph{Markov} martingale $M=(M_t)_{t \geq 0}$ such that $M_t\sim \mu_t$ for all $t\geq0$.  
Several proofs and ramifications of Kellerer's theorem can be found in the literature. Hirsch and Roynette~\cite{HiRo12} construct martingales as solutions of stochastic differential equations and use  an approximation argument. Lowther~\cite{Lo08,Lo09} shows that under some regularity assumptions  there exists an ACD martingale with marginals $(\mu_t)_{t \geq 0}$.
Here, ACD stands for ``almost-continuous diffusion'', a condition implying the strong Markov
property and stochastic continuity.
Beiglb\"ock, Huesmann and Stebegg~\cite{BeHuSt16} use a certain solution of the Skorokhod problem, which is Lipschitz-Markov, to construct a martingale which is Markov.
The recent book by Hirsch, Profeta, Roynette, and Yor~\cite{HiPrRoYo11} contains a wealth of constructions
of peacocks and associated martingales.

%While there are many works that aim at producing martingales with additional properties,
%we extend Strassen's theorem in a different direction.
The main question that we consider in this paper is the following: given $\epsilon>0$, a metric~$d$ on~$\setM$ -- the set of all probability measures on~$\setR$ with finite mean -- and a sequence of measures $(\mu_t)_{t \in T}$ in $\setM$, when does a sequence $(\nu_t)_{t \in T}$ in $\setM$ exist,
such that $d(\mu_t, \nu_t) \leq \epsilon$ and such that the sequence $(\nu_t)_{t \in T}$ is a peacock?  Here $T$ is either $\setN$ or the interval $[0,1]$.  Once we have constructed a peacock, we know, from the results mentioned above, that there is a martingale (with certain properties) with these marginals. We thus want to find out when there is a martingale $M$ such that the law of $M_t$ is close to $\mu_t$ for all~$t$. 
We will state necessary and sufficient conditions when $d$ is the infinity Wasserstein distance, the stop-loss distance,
the Prokhorov distance, and the L\'evy distance.

The infinity Wasserstein distance is a natural analogue of the $p$-Wasserstein distance.
Besides the dedicated probability metrics literature (e.g., \cite{Ra91,RaRuSc92}),
it has been studied in an optimal transport setting in~\cite{ChDeJu08}. 
It also has applications in graph theory, where it is referred to as the bottleneck distance (see p.~216 of~\cite{EdHa10}).
%Therefore we will give a formal definition of the infinity Wasserstein distance, and will prove that
%it is indeed a metric.
A well-known alternative representation of the infinity Wasserstein distance shows some
similarity to the L\'{e}vy distance. The stop-loss distance was introduced by Gerber in~\cite{Ge79} and has been studied in actuarial science (see for instance~\cite{DeDh92,KaVaGo88}).

For both of these metrics, we translate existence of a peacock within $\epsilon$-distance
into a more tractable condition: There has to exist a real number (with the interpretation
of the desired peacock's mean) that satisfies a countable collection
of finite-dimensional conditions, each explicitly expressed in terms
of the call functions $x\mapsto\int (y-x)^+  \mu_t(dy)$ of the given sequence of measures. For the infinity Wasserstein distance,
the existence proof is not constructive, as it uses Zorn's lemma.
For the stop-loss distance, the problem is much simpler, and our proof is short and
constructive. Note, though, that the result about the infinity Wasserstein distance admits
a financial application, which was the initial motivation for this work.
The problem is similar to the one considered by Davis and Hobson~\cite{DaHo07}: given a set of European call option prices with different maturities on the same underlying, we want to know when there is a model which is consistent with these prices. In contrast to Davis and Hobson we allow a bid-ask spread, bounded by some constant, on the underlying. This application will be
developed in the companion paper~\cite{GeGu16}.

Our proof approach is similar for both metrics: we will construct minimal and maximal elements (with respect to the convex order) in closed balls, and then use these elements to derive our conditions. In the case
of the infinity Wasserstein distance, we will make use of the lattice structure of certain subsets of closed balls.

The L\'evy distance was first introduced by L\'evy in 1925 (see~\cite{Le25}). Its importance is partially due to the fact that~$d^\mathrm{L}$ metrizes weak
convergence of measures on~$\setR$. 
The Prokhorov distance, first introduced in~\cite{Pr56}, is a metric on measures on an arbitrary separable metric space, and is often referred to as a generalization of the L\'evy metric, since $d^\mathrm{P}$ metrizes weak convergence on any separable metric space. 
For these two metrics, peacocks within $\epsilon$-distance \emph{always} exist,
and can be explicitly constructed.

The definition of the infinity Wasserstein distance yields a coupling representation, and it is a natural
question whether -- assuming the existence of a peacock nearby -- there is a filtered probability space on which the coupling can be
realized with a martingale.
For finite sequences of measures with finite support, we show in Theorem~\ref{thm:strassenW}
that this is true, thus extending (a special case of) Strassen's theorem in a novel direction.

%This leads to definition of $\epsilon$-consistent models. We will state model independent arbitrage strategies in case there is no $\epsilon$-consistent model. Cousot~\cite{Co07} considers the same setting as Davis and Hobson but allows options
%(and not the underlying) to have different bid and ask prices. 

The structure of the paper is as follows. Section~\ref{notation} specifies our notation and introduces the most important definitions. 
Section~\ref{infsec} contains our main result on approximation by peacocks
using the infinity Wasserstein distance. Its proof is given in Section~\ref{se:pr ram},
and a continuous time version
can be found in Section~\ref{se:cont}.
In Section~\ref{sec:SL} we will treat the approximation problem for the stop-loss distance.
After collecting some well-known facts on the L\'evy and Prokhorov distances in Section~\ref{sec:LP},
we will prove a criterion for approximation by peacocks under these metrics in
Section~\ref{se:thm P L}. Section~\ref{se:new Strassen} presents our novel
extension of (a special case of) Strassen's theorem, and some related open problems that we propose to tackle in future work.

\section{Notation and preliminaries} \label{notation}

Let $\setM$ denote the set of all probability measures on $\setR$ with finite mean. We start with the definition of convex order.
\begin{definition}\label{def:conv}
Let $\mu, \nu$ be two measures in $\setM$. Then we say that $\mu$ is smaller in convex order
than $\nu$, in symbols $\mu \leqc \nu$, if for every convex function $\phi: \setR \rightarrow \setR$ we have
$\int \phi \; d \mu \leq \int \phi \; d \nu$, whenever both integrals are
finite.\footnote{The apparently stronger requirement 
that the inequality $\int \phi \, d \mu \leq \int \phi \, d \nu$ holds 
for convex~$\phi$ whenever it
makes sense, i.e., as long as both sides exist in $(-\infty,\infty]$, leads to an
equivalent definition.
This can be seen by the following argument, similar to Remark~1.1 in~\cite{HiPrRoYo11}: Assume that
the inequality holds if both sides are finite, and let~$\phi$ (convex)
 be such that $\int \phi \, d \mu=\infty$.
We have to show that then $\int \phi \; d \nu=\infty$. Since~$\phi$ is the envelope of the
affine functions it dominates, we can find convex~$\phi_n$ with $\phi_n\uparrow \phi$ pointwise,
and such that each~$\phi_n$ is $C^2$ and $\phi_n''$ has compact support. By monotone convergence,
we then have $\int \phi \, d \nu= \lim \int \phi_n \, d \nu \geq
\lim \int \phi_n \, d \mu=\int \phi \; d \mu=\infty$. 
%NEU
Note that that the convexity of $\phi$ guarantees that $\int \phi \, d \nu>-\infty$.} 
A family of measures $(\mu_t)_{t \in T}$ in $\setM$, where $T \subseteq [0,\infty)$, is called peacock, if $\mu_s \leqc \mu_t$ for all $s \leq t$ in~$T$ (see Definition~1.3 in~\cite{HiPrRoYo11}). 
\end{definition}
Intuitively, $\mu \leqc \nu$ means that~$\nu$ is more dispersed than~$\mu$, as convex integrands
tend to emphasize the tails.
By choosing $\phi(x)=x$ resp.\ $\phi(x)=-x$, we see that $\mu \leqc \nu$ implies that $\mu$ and $\nu$ have the same mean. 
As mentioned in the introduction, Strassen's theorem asserts the following:
\begin{thm}\label{thm:str}
(Strassen~\cite{St65})
  For any peacock with $T=\mathbb N$, there is a martingale whose family of one-dimensional marginal laws coincides with it.
\end{thm}
The converse implication is of course true as well, as a trivial consequence of Jensen's
inequality. As mentioned in the introduction, the equivalence also holds for time index
set $T=\mathbb R^+$ \cite{HiRo12,Ke72,Ke73}.
For $\mu \in \setM$ and $x \in \setR$ we define
\[
  R_\mu(x)= \int\nolimits_\setR (y-x)^+  \mu(dy) 
  \quad \text{and} \quad F_\mu(x)= \mu\bigl((-\infty,x]\bigr).
\]
We call $R_\mu$ the call function of $\mu$, as in financial terms it is the (undiscounted)
price of a call option with strike~$x$, written on an underlying with risk-neutral law~$\mu$ at maturity.
(It is also known as integrated survival function~\cite{MuRu01}.)
The mean of a measure~$\mu$ will be denoted by $\mathbb{E} \mu = \int y\, \mu(dy)$. 
The following proposition summarizes important properties of call functions.
All of them are well known. In particular, the equivalence in part~\ref{convexorderR} has been
used a lot to investigate the convex order; see, e.g., \cite{MuSt02}.
\begin{prop} \label{RFprop}
Let $\mu,\nu$ be two measures in $\setM$. Then:
\begin{enumerate}[label=(\roman*)]
\item \label{convexcall} $R_\mu$ is convex, decreasing and strictly decreasing on $\{R_\mu >0 \}$. Hence the right derivative of $R_\mu$ always exists and is denoted by $R'_\mu$.
\item \label{limitcall} $\lim_{x \rightarrow \infty} R_\mu(x)=0$ and $\lim_{x \rightarrow -\infty} (R_\mu(x)+x)=\mathbb{E} \mu$. In particular, if $\mu([a,\infty))=1$ for $a >- \infty$, then $\mathbb{E} \mu= R_\mu(a)+a$. 
\item \label{RFrel} $R'_\mu(x)=-1+F_\mu(x)$ and $R_\mu(x)= \int\nolimits_x^\infty (1-F_\mu(y)) \, dy$, for all $x \in \setR$.
\item \label{convexorderR} $\mu \leqc \nu$ holds if and only if $R_\mu(x) \leq R_\nu(x)$  for all $x \in \setR$ and  $\mathbb{E} \mu = \mathbb{E} \nu$.
\item \label{it:fund thm}
  For $x_1 \leq x_2\in\setR$, we have $R_\mu(x_2) - R_\mu(x_1)= \int_{x_1}^{x_2} R'_\mu(y)\, dy$.
\end{enumerate} 
Conversely, if a function $R: \setR \rightarrow \setR$ satisfies \ref{convexcall} and \ref{limitcall}, then there exists a probability measure $\mu \in \setM$ with finite mean such that $R_\mu =R$. 
%The set of these functions will be denoted by $\setC$. 
\end{prop}
%\begin{proof}
  As for~\ref{it:fund thm}, note that~$R_\mu'$ is increasing, thus integrable, and that
  the fundamental theorem of calculus holds for right derivatives. See~\cite{BoGo86} for a short proof.
  The other assertions of Proposition~\ref{RFprop}
  are proved in~\cite{HiRo12}, Proposition 2.1, and~\cite{HiPrRoYo11}, Exercise 1.7.
%\end{proof}
%
For a metric $d$ on $\setM$, denote by $B^d (\mu, \epsilon)$ the closed ball with respect to  $d$, with center $\mu$ and radius $\epsilon$. Then our main question is:
\begin{pro}\label{pr:main}
  Given $\epsilon>0$, a metric~$d$ on~$\setM$, and
  a sequence $(\mu_n)_{n\in\setN}$ in~$\mathcal{M}$, when does there exist
a peacock $(\nu_n)_{n\in\setN}$ with $\nu_n\in B^d (\mu_n, \epsilon)$ for all~$n$?
\end{pro}
Note that this can also be phrased as
\[
  d_\infty\bigl((\mu_n)_{n\in\setN},(\nu_n)_{n\in\setN}\bigr)\leq\epsilon,
\]
where
\[
  d_\infty\bigl((\mu_n)_{n\in\setN},(\nu_n)_{n\in\setN}\bigr) = \sup_{n\in\setN} d(\mu_n,\nu_n)
\]
defines a metric on $\setM^\setN$ (with possible value infinity).
For some results on this kind of infinite product metric,
%for bounded metric spaces though,
we refer to~\cite{Bo91}.
Clearly, a solution to Problem~\ref{pr:main} settles the case of finite sequences
$(\mu_n)_{n=1,\dots,n_0}$, too, by extending the sequence with $\mu_n:=\mu_{n_0}$ for $n> n_0$.

To fix ideas, consider the case where the given sequence $(\mu_n)_{n=1,2}$ has only two elements.
We want to find measures $\nu_n \in B^d(\mu_n, \epsilon)$, $n=1,2$, such that $\nu_1 \leqc \nu_2$.
Intuitively, we want 
$\nu_1$ to be as small as possible and $\nu_2$ to be as large as possible, in the convex order.
Recall that a peacock has constant mean, which is fixed as soon as~$\nu_1$ is chosen.  We will denote
the set of probability measures on $\setR$ with mean $m\in\setR$ by $\setM_{m}$. 
These considerations lead us to the following problem.
\begin{pro} \label{probSLgen} 
Suppose that a metric $d$ on $\setM$, a measure $\mu \in \setM$ and two numbers $\epsilon>0$
and $m\in\mathbb R$ are given. When are there two measures $\mu^{\min}, \mu^{\max} \in B^{d}(\mu, \epsilon) \cap \setM_{m}$ such that
\begin{equation*}
\mu^{\min} \leqc \nu \leqc \mu^{\max}, \quad \mbox{for all} \ \nu \in B^{d}(\mu, \epsilon) \cap \setM_{m}\ ?
\end{equation*}
\end{pro}
We now recall the definition of the infinity
Wasserstein distance\footnote{The name ``\emph{infinite} Wasserstein distance'' is also in use,
but ``\emph{infinity} Wasserstein distance'' seems to make more sense
(cf.\ ``\emph{infinity} norm'').} $W^\infty$, and its connection to call
functions.
\begin{definition} \label{infwassersteindistance}
The infinity Wasserstein distance is the mapping
$W^\infty: \setM \times \setM \rightarrow [0,\infty]$ defined by
\begin{equation*}
W^\infty (\mu,\nu) = \inf \left\|X-Y \right\|_\infty,
\end{equation*} 
where the infimum is taken over all probability spaces $(\Omega,\mathcal{F}, \pp)$ 
and random pairs $(X,Y)$ with marginals given by $\mu$ and $\nu$.
\end{definition}
For various other probability metrics and their relations, see~\cite{GiSu02,Ra91}.
We will use the words ``metric'' and ``distance'' for
mappings $\setM\times\setM \to [0,\infty]$ in a loose sense.
Since all our results concern \emph{concrete} metrics, there is no need to give
a general definition (as, e.g., Definition~1 in Zolotarev~\cite{Zo76}).
 The metric $W^\infty$ has the following representation in terms of call functions (see, e.g., \cite{Li02}, p.~127):
\begin{equation}\label{eq:W repr}
  W^\infty (\mu,\nu) = \inf \Bigl\{h >0: R'_\mu(x-h)
    \leq R'_\nu(x) \leq R'_\mu(x+h),\ \forall x \in \setR \Bigr\}.
\end{equation}
By~\eqref{eq:W repr} and Proposition~\ref{RFprop}~\ref{RFrel}, $W^\infty$ can also be written as
\[
  W^\infty (\mu,\nu) = \inf \Bigl\{h >0: F_\mu(x-h)
    \leq F_\nu(x) \leq F_\mu(x+h),\ \forall x \in \setR \Bigr\}.
\]
%
%\begin{rem} In the following we will only work with measures which have a $W^\infty$-distance smaller than a %given $\epsilon >0$. So we could replace $W^\infty$ with
%$$
%\widetilde W^\infty(\mu,\nu)= \min \bigl\{W^\infty(\mu,\nu), C \bigr\},
%$$
%for a constant $C>\epsilon$ and would still get the same results, working
%with a metric in the classical sense. Alternatively we might restrict the
%domain of $W^\infty$ to measures with bounded support to get a metric in the classical sense. 
%\end{rem}
%
We will see below (Proposition~\ref{LSwinf}) that, when~$d$ is the infinity Wasserstein distance,
Problem~\ref{probSLgen}  has a solution $(\mu^{\min},\mu^{\max})$ if and only if
$|m-\mathbb{E}\mu|\leq\epsilon$.
As an easy consequence, given $(\mu_n)_{n=1,2}$,
the desired ``close'' peacock  $(\nu_n)_{n=1,2}$ exists if and only if there is
an~$m$ with $|m-\mathbb{E}\mu_1|\leq\epsilon$, $|m-\mathbb{E}\mu_2|\leq\epsilon$
such that the corresponding measures $\mu_1^{\min},\mu_2^{\max}$ satisfy $\mu_1^{\min}\leqc \mu_2^{\max}$.
Then, $(\nu_1,\nu_2)=(\mu_1^{\min},\mu_2^{\max})$ is a possible choice.

Besides the infinity Wasserstein distance, we will solve Problems~\ref{pr:main} and~\ref{probSLgen} also
for the stop-loss distance (Proposition~\ref{LSdsl}),
%For both metrics, we not only deal with two-element sequences, but solve our main problem 
%of existence of peacocks within $\epsilon$-distance,
for index sets~$\setN$ and $[0,1]$
(see Theorems~\ref{strassenwinf}, \ref{thm:01}, \ref{mainsl}, and~\ref{thm:sl01}).
For the L\'evy distance and the Prokhorov distance we will use different techniques and solve
Problem~\ref{pr:main} for index set~$\setN$
(see   Corollary~\ref{cor:strassendp} and Theorem~\ref{thm: dpLstrassen}).

\section{Approximation by peacocks: infinity Wasserstein distance (discrete time)} \label{infsec}
%\section{Strassen's theorem for the infinity Wasserstein distance: discrete time} \label{infsec}

We now start to investigate the interplay between the infinity Wasserstein distance and the convex order.
Recall that $\setM_m$ denotes the set of probability measures on~$\mathbb R$ with mean~$m$.
It is a well known fact that the ordered set $(\setM_m, \leqc)$ is a lattice for all $m \in \setR$,
with least element~$\delta_m$ (Dirac delta). See for instance \cite{KeRo00,MuSc06}.
The lattice property means that,  given any two measures $\mu, \nu \in \setM_m$, there is a unique supremum, denoted by $\mu \vee \nu$, and a unique infimum, denoted by $\mu \wedge \nu$, with respect to convex order. It is easy to prove that the corresponding call functions are
 $R_{\mu \vee \nu}=R_\mu\vee R_\nu$ and $R_{\mu \wedge \nu}=\conv(R_\mu,R_\nu)$. Here and in the following $\conv(R_\mu,R_\nu)$ denotes the convex hull  of $R_\mu$ and $R_\nu$, i.e., the largest convex function that is majorized by $R_\mu \wedge R_\nu$. 

In the following we will denote balls with respect to $W^\infty$ by $B^\infty$. The next lemma shows that $(B^\infty(\mu,\epsilon) \cap \setM_m, \leqc)$ is a sublattice of  $(\setM_m, \leqc)$, which will be important afterwards. Recall that two measures can be comparable w.r.t.\ convex order only if their means agree.
This accounts for the relevance of sublattices of the form
$(B^\infty(\mu,\epsilon) \cap \setM_m, \leqc)$ for our problem: If a peacock $(\nu_n)_{n\in\setN}$
satisfying $\nu_n\in B^\infty(\mu_n,\epsilon)$ for all $n\in\setN$ exists, then we necessarily have
$\nu_n\in B^\infty(\mu_n,\epsilon)\cap \setM_m$, $n\in\setN$, with
$\mathbb{E}\nu_1=\mathbb{E}\nu_2=\dots=m$.

\begin{lem} \label{mininkugel}
%NEU
  Let $\epsilon>0$ and $\mu, \nu_1, \nu_2 \in \setM$ such that $\mathbb{E}\nu_1=\mathbb{E}\nu_2=m$. 
  Then if $\nu_1, \nu_2 \in B^\infty(\mu, \epsilon)\cap \setM_{m}$ we have 
  $\nu_1 \vee \nu_2 \in B^\infty(\mu, \epsilon)\cap \setM_{m}$ 
  and $\nu_1 \wedge \nu_2 \in B^\infty(\mu, \epsilon)\cap \setM_{m}$. 
\end{lem}

\begin{proof}
Denote the call functions of $\nu_1$ and $\nu_2$ with $R_1$ and $R_2$.
We start with $\nu_1 \vee \nu_2$. It is easy to check that $R: x \mapsto R_1(x)\vee R_2(x)$ is a call function satisfying $R'(x) \in \{R_1'(x), R_2'(x)\}$ for all~$x\in\setR$.
By Proposition~\ref{RFprop}~\ref{limitcall}, it is also clear that $\nu_1 \vee \nu_2\in \setM_m$.
This proves the assertion. 

As for the infimum, we will first assume that there exists $x_0 \in \setR $ such that $R_1(x) \leq R_2(x)$ for $x \leq x_0$ and $R_2(x) \leq R_1(x)$ for $x \geq x_0.$ 
Then there exist $x_1 \leq x_0$ and $x_2 \geq x_0$ such that the convex hull of $R_1$ and $R_2$ can be written as (see~\cite{Ob07})
\begin{equation*}
\mathrm{conv}(R_1,R_2)(x) = \begin{cases}R_1(x), & x \leq x_1,  \\
R_1(x_1)+\frac{R_2(x_2)-R_1(x_1)}{x_2-x_1} (x-x_1), & x \in [x_1,x_2], \\
R_2(x), & x\geq x_2. \end{cases}
\end{equation*}
Now observe that for all $x \in [x_1, x_2)$
\begin{align*}
R_\mu'(x-\epsilon) &\leq R_2'(x) \leq  R_2'(x_2-) \\
 &\leq \frac{R_2(x_2)-R_1(x_1)}{x_2-x_1} \\
 &\leq R_1'(x_1) \leq R_1'(x) \leq R_\mu'(x+\epsilon),
\end{align*}
and hence $\mathrm{conv}(R_1,R_2)'(x) \in [R_\mu'(x-\epsilon), R_\mu'(x+\epsilon)]$. Therefore $\nu_1 \wedge \nu_2 \in B^\infty(\mu, \epsilon)\cap \setM_{m}$. 

For the general case, note that for all $x \in \setR$ we have by~\cite{Ob07} that either $\mathrm{conv}(R_1,R_2)(x)= R_\mu(x) \wedge R_\nu(x)$, or that~$x$ lies in an interval $I$ such that $\mathrm{conv}(R_1,R_2)$ is affine on $I$. If the latter condition is the case, then we can derive bounds for the right-derivative $\mathrm{conv}(R_1,R_2)'(x), x \in I$, exactly as before. The situation is clear if 
either $\mathrm{conv}(R_1,R_2)(x)=R_1(x)$ or $\mathrm{conv}(R_1,R_2)(x)=R_2(x)$. 
\end{proof}
We now show that the sublattice $(B^\infty(\mu,\epsilon) \cap \setM_m, \leqc)$ contains a least and a greatest element with respect to the convex order. This is the subject of the following proposition,
which solves Problem~\ref{probSLgen} for the infinity Wasserstein distance.
As for the assumption~$m \in [\mathbb{E}\mu-\epsilon,\mathbb{E}\mu+\epsilon]$
in Proposition~\ref{LSwinf}, it is necessary to ensure
that $B^\infty(\mu,\epsilon) \cap \setM_m$ is not empty. Indeed, if $W^\infty(\mu_1,\mu_2)\leq \epsilon$
for some $\mu_1,\mu_2\in\setM$, then by~\eqref{eq:W repr}, Proposition~\ref{RFprop} \ref{limitcall},
\ref{it:fund thm}, and the continuity of call functions, we obtain
\begin{equation}\label{eq:R est}
  R_{\mu_1}(x+\epsilon) \leq R_{\mu_2}(x) \leq R_{\mu_1}(x-\epsilon), \quad x\in\setR.
\end{equation}
By part~\ref{limitcall} of Proposition~\ref{RFprop}, it follows that
$|\mathbb{E}\mu_1 - \mathbb{E}\mu_2|\leq \epsilon$.
\begin{prop}\label{LSwinf}
Given $\epsilon>0$, a measure $\mu \in \setM$ and $m \in [\mathbb{E}\mu-\epsilon,\mathbb{E}\mu+\epsilon]$, there exist unique measures $S(\mu), T(\mu) \in B^\infty(\mu,\epsilon) \cap \setM_m$ such that
\begin{equation*}
S(\mu) \leqc \nu \leqc T(\mu) \quad \text{for all} \ \nu \in B^\infty(\mu,\epsilon) \cap \setM_m.
\end{equation*}
The call functions of $S(\mu)$ and $T(\mu)$ are explicitly given by
\begin{align}
R^{\min}_\mu(x)&=R_{S(\mu)}(x)=\Bigl(m+R_\mu(x-\epsilon)-\bigl(\mathbb{E}\mu+\epsilon \bigr) \Bigr) \vee R_\mu(x+\epsilon), \label{eq:Rmin} \\
R_\mu^{\max}(x)&=R_{T(\mu)}(x) = \conv\Bigl(m+R_\mu(\cdot +\epsilon)-\bigl(\mathbb{E}\mu-\epsilon \bigr) ,R_\mu(\cdot -\epsilon) \Bigr)(x). \label{eq:Rmax}
\end{align}
To highlight the dependence on $\epsilon$ and $m$ we will sometimes write $S(\mu; m, \epsilon)$ and  $R^{\min}_\mu(\: \cdot \:; m, \epsilon)$, respectively
$T(\mu; m, \epsilon)$ and  $R^{\max}_\mu(\: \cdot \:; m, \epsilon)$.
\end{prop}
\begin{proof}
We define $R_\mu^{\min}$ and $R_\mu^{\max}$ by the right hand sides of~\eqref{eq:Rmin}
resp.~\eqref{eq:Rmax}, and argue that the associated measures $S(\mu)$ and $T(\mu)$ have the stated property.
Clearly $R_\mu^{\min}$ is a call function, and we have 
\begin{align*}
\mathbb{E}R_\mu^{\min} &= \lim_{x \rightarrow -\infty} 
  \Bigl( m+R_\mu(x-\epsilon)-\bigl(\mathbb{E}\mu+\epsilon \bigr) +x \Bigr)
  \vee\Bigl(R_\mu(x+\epsilon)+x\Bigr) \\
  &= m \vee \bigl(\mathbb{E}\mu - \epsilon\bigr)=m.
\end{align*}
{}From the convexity of $R_\mu$ we can deduce the existence of $v \in \setR \cup \{\pm \infty\}$ such that
\[
R_\mu^{\min}(x) = \begin{cases}m+R_\mu(x-\epsilon)-\bigl(\mathbb{E}R_\mu+\epsilon \bigr), & x \leq v, \\ 
                    R_\mu(x+\epsilon)  & x \geq v. \\
									\end{cases}
\]
Hence we get that $(R_\mu^{\min})'(x) \in [R_\mu'(x-\epsilon), R_\mu'(x+\epsilon)]$ for all~$x$.
By~\eqref{eq:W repr}, the measure associated with
$R_\mu^{\min}$ lies in $B^\infty(\mu,\epsilon) \cap \setM_m$. To the left of~$v$,
$R_\mu^{\min}$ is as steep as possible (where steepness refers to the absolute value of the
right derivative), and to the right of~$v$ it is as flat as possible
(see Figure~\ref{fig:rmin}).
\begin{figure}[ht]
   \begin{center}
      \includegraphics[width=10cm]{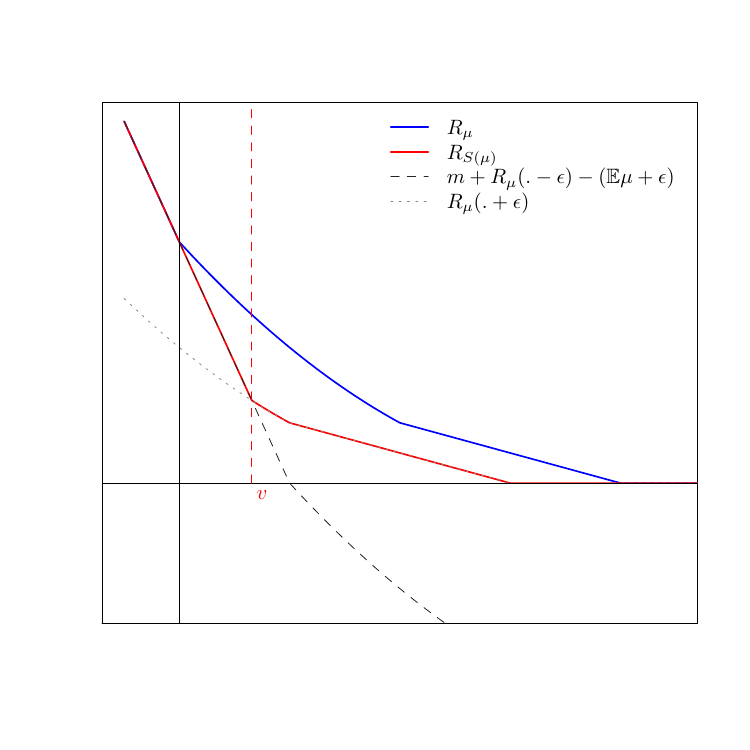}
   \caption{Illustration of the function $R_\mu^{\min}=R_{S(\mu)}$ (lower solid curve): to the left of~$v$, $R_{S(\mu)}$ is as steep as possible and to the right of~$v$, $R_{S(\mu)}$ is as small as possible.}
     \label{fig:rmin}
   \end{center}
\end{figure}
From this and convexity,
it is easy to see that~$S(\mu)$ is a least element.

Similarly we can show that $\mathbb{E}R_\mu^{\max}=m$, and thus it suffices to show that
\[
  (R_\mu^{\max})'(x) \in [R_\mu'(x-\epsilon), R_\mu'(x+\epsilon)].
\]
But this can be done exactly as in Lemma \ref{mininkugel}.
\end{proof}

\begin{rem}
It is not hard to show that 
\begin{equation*} 
R_\mu^{\max}(x)= \begin{cases}m+R_\mu(x+\epsilon)-\bigl(\mathbb{E}\mu-\epsilon \bigr), & x  \leq x_1, \\[1mm] 
                   R_\mu(x_1+\epsilon) + \frac{\bigl(\mathbb{E}\mu-\epsilon \bigr)-m}{2\epsilon} \bigl(x-x_1-2\epsilon),  & x \in [x_1, x_1+2\epsilon],\\
									R_\mu(x-\epsilon), & x \geq x_1+2\epsilon,\end{cases} \\
\end{equation*}
where $$
x_1=  \inf \biggl\{x \in \setR : R_\mu'(x+\epsilon)\geq- \frac{m-\bigl(\mathbb{E}\mu-\epsilon \bigr)}{2\epsilon} \biggr\}.$$
\end{rem}
Before formulating our first main theorem, we recall that a peacock is a sequence
of probability measures with finite mean and increasing w.r.t.\ convex order (Definition~\ref{def:conv}).
We now give a simple reformulation of this property.
For a given sequence of call functions~$(R_n)_{n\in\setN}$, define, for $N\in\setN$ and
$x_1,\dots,x_N\in\setR$,
\begin{equation}\label{eq:Phi 0}
  \Phi_N(x_1,\dots,x_N) = R_1(x_1) + \sum_{n=2}^N\Big(R_n(x_n)-R_n(x_{n-1})\Big) - R_{N+1}(x_N).
\end{equation}
\begin{prop}\label{prop:consistent}
  A sequence of call functions $(R_n)_{n\in\setN}$ with constant mean defines a peacock if and only if
  $\Phi_N(x_1,\dots,x_N) \leq 0$ for all $N\in\setN$ and $x_1,\dots,x_N\in\setR$.
\end{prop}
\begin{proof}
  According to Proposition~\ref{RFprop}~\ref{convexorderR}, we need to check whether the sequence
  of call functions increases. Let $n\in\setN$ be arbitrary. If we set the $n$-th component
  of $(x_1,\dots,x_{n+1})$ to an arbitrary $x\in\setR$ and let all others tend to $\infty$, we get
  \[
    \Phi_{n+1}(\infty,\dots,\infty,x,\infty) = R_n(x) - R_{n+1}(x).
  \]
  The sequence of call functions thus increases, if $\Phi$ is always non-positive. Conversely,
  assume that $(R_n)_{n\in\setN}$ increases. Then, for $N\in\setN$ and $x_1,\dots,x_N\in\setR$,
  \begin{align*}
     \Phi_N(x_1,\dots,x_N) &\leq  R_1(x_1) + \sum_{n=2}^N R_{n+1}(x_n)-
       \sum_{n=2}^N R_n(x_{n-1}) - R_{N+1}(x_N) \\
     &= R_1(x_1) + \sum_{n=3}^{N+1} R_{n}(x_{n-1})-
       \sum_{n=2}^N R_n(x_{n-1}) - R_{N+1}(x_N) \\
     &= R_1(x_1) - R_2(x_1) \leq 0.
  \end{align*}
\end{proof}
We now extend the definition of~$\Phi_N$ for $x_1,\dots,x_N\in\setR$, %\cup \{\pm\infty\}$,
 $m\in\setR$, and $\epsilon>0$ as follows, using the notation from Proposition~\ref{LSwinf}:
\begin{multline}\label{eq:Phi}
  \Phi_N(x_1,\dots,x_N;m,\epsilon)=
  R_1^{\min}(x_1; m, \epsilon) \\
  + \sum_{n=2}^{N}\Big(R_n(x_n+\epsilon \sigma_n )-R_n(x_{n-1}+\epsilon \sigma_n )\Big)
   - R_{N+1}^{\max}(x_N; m, \epsilon).
\end{multline}
Here, $R_1^{\min}$ is the call function of $S(\mu_1; m, \epsilon)$, $R_{N+1}^{\max}$ is the call function of $T(\mu_{N+1}; m ,\epsilon)$, and
\begin{equation}\label{eq:sigma}
  \sigma_n= \sgn(x_{n-1}-x_n) = \begin{cases}
    1, &\text{if} \ x_{n-1}> x_n, \\
		0, &\text{if} \ x_{n-1} = x_n, \\
    -1, &\text{if} \ x_{n-1}< x_n, \\
  \end{cases}
\end{equation}
depends on~$x_{n-1}$ and~$x_n$.
%Moreover, we set $R_n(\infty)=0$ for all $n \in \setN$ and 
%$$
%R_n(-\infty \pm \epsilon)-R_{n+1}(-\infty \pm \epsilon):= \lim_{x\rightarrow -\infty} R_n(x \pm \epsilon)-R_{n+1}(x \pm \epsilon).%= \mathbb{E}\mu-\bigl(\mathbb{E}\nu+\epsilon\bigr).
%$$
%
Clearly, for $\epsilon=0$ and $\mathbb{E}\mu_1 = \mathbb{E}\mu_2 = \dots = m$, we recover~\eqref{eq:Phi 0}:
\begin{equation}\label{eq:Phi consistent}
   \Phi_N(x_1,\dots,x_N;m,0) = \Phi_N(x_1,\dots,x_N), \quad N \in \setN,\
   x_1,\dots,x_N \in \setR.
\end{equation}
The following theorem gives an equivalent condition for the existence of a peacock
within $W^\infty$-distance $\epsilon$ of a given sequence of measures, thus solving
Problem~\ref{pr:main} for the infinity Wasserstein distance, and is our first
main result.
%By Proposition~\ref{prop:consistent} and~\eqref{eq:Phi consistent},
%it is consistent with Strassen's theorem (Theorem~\ref{thm:str}), i.e., the
%case $\epsilon=0$.
Note that the functions~$\Phi_N$ defined in~\eqref{eq:Phi} have explicit expressions in terms of the
given call functions, as $R^{\min}$ and $R^{\max}$ are explicitly given
by~\eqref{eq:Rmin} and~\eqref{eq:Rmax}. The existence criterion we obtain is thus rather explicit;
the existence \emph{proof} is not constructive, though, as mentioned in the introduction.
%Moreover, note that we use Strassen's theorem in the proof; for $\epsilon=0$, our proof
%reduces to a triviality, and not to a proof of Strassen's theorem.
%
\begin{thm} \label{strassenwinf}
   Let $\epsilon>0$ and $(\mu_n)_{n \in \setN}$ be a sequence in $\setM$ such that 
   \[
     I:= \bigcap_{n \in \setN}[\mathbb{E}\mu_n-\epsilon, \mathbb{E}\mu_n+\epsilon] 
   \]
   is not empty. Denote by $(R_n)_{n \in \setN}$ the corresponding call functions, and
   define~$\Phi_N$ by~\eqref{eq:Phi}. Then there exists a peacock $(\nu_n)_{n \in \setN}$ 
   such that 
   \begin{equation} \label{eq:Winf}
     W^\infty(\mu_n, \nu_n) \leq \epsilon, \quad \mbox{for all} \ n \in \setN,
   \end{equation}
   if and only if for some $m \in I$
   and for all $N \in \setN$,  $x_1, \dots, x_N \in \setR$, %\cup\{\pm \infty\}$, 
   we have
   \begin{equation} \label{minklmaxwinf}
     \Phi_N(x_1,\dots,x_N;m,\epsilon)\leq0.
   \end{equation}
   In this case it is possible to choose $\mathbb{E}\nu_1=\mathbb{E}\nu_2=\dots=m$. 
\end{thm}
%Note that \eqref{minklmaxwinf} is a generalization of \ref{convexorderR} in Proposition \ref{RFprop}:
%indeed if $\epsilon=0$ then $S(\mu_n)=T(\mu_n)=\mu_n$ for all $n \in \setN$ and by appropriate
%choice of the $x_1, \dots,x_N$, \eqref{minklmaxwinf} simplifies to: $R_l \leq R_n$ for $l \leq n$.
%
The proof of Theorem~\ref{strassenwinf} is given in Section~\ref{se:pr ram},
building on Theorem~\ref{minconstruct} and Corollary~\ref{xycor} below. 
In view of our intended application (see~\cite{GeGu16}), we now give an alternative
formulation of Theorem~\ref{strassenwinf}, which avoids the existential
quantification ``for some $m\in I$''.
%There is alternative formulation of Theorem~\ref{strassenwinf} which is more usefull for applications.
%Note that the conditions stated in Corollary~\ref{strassenwinf2} do not depend on a particular choice of~$m \in I$.
Note that the expressions inside the suprema
in \eqref{eq:minklmax1}-\eqref{eq:minklmax3} are similar to~$\Phi_N$, defined in~\eqref{eq:Phi}.
Corollary~\ref{strassenwinf2} is proved towards the end of Section~\ref{se:pr ram}.
\begin{cor} \label{strassenwinf2}
	Let $\epsilon>0$ and $(\mu_n)_{n \in \setN}$ be a sequence in $\setM$ such that 
	\[
	I:= \bigcap_{n \in \setN}[\mathbb{E}\mu_n-\epsilon, \mathbb{E}\mu_n+\epsilon] 
	\]
	is not empty. Denote by $(R_n)_{n \in \setN}$ the corresponding call functions.
	Then there exists a peacock $(\nu_n)_{n \in \setN}$ 
	such that~\eqref{eq:Winf} holds 
	if and only if 
	\begin{align} \label{eq:minklmax1}
	\sup_{\substack{N_1 \in \mathbb{N} \\ x_1, \dots, x_{N_1} \in \setR}} & \biggl\{ 
	R_1(x_1 + \epsilon) + \sum_{n=2}^{N_1} R_n(x_n + \epsilon \sigma_n) - R_{n}(x_{n-1} + \epsilon \sigma_n) - R_{N_1+1}(x_N - \epsilon) \biggr\} \leq 0, \\
	\label{eq:minklmax2}
	\sup_{\substack{N_1 \in \mathbb{N} \\ x_1, \dots, x_{N_1} \in \setR}} & \biggl\{ 
	R_1(x_1 - \epsilon) + \sum_{n=2}^{N_1} R_n(x_n + \epsilon \sigma_n) - R_{n}(x_{n-1} + \epsilon \sigma_n) - R_{N_1+1}(x_N + \epsilon) + \mathbb{E}\mu_{N_1+1} -\mathbb{E}\mu_{1} \biggr\} \leq 2 \epsilon, \\
	\nonumber \sup_{\substack{N_1 \in \mathbb{N} \\ x_1, \dots, x_{N_1} \in \setR}} & \biggl\{ 
	 R_1(x_1 + \epsilon) + \sum_{n=2}^{N_1} R_n(x_n + \epsilon \sigma_n) - R_{n}(x_{n-1} + \epsilon \sigma_n) - R_{N_1+1}(x_{N_1} + \epsilon) + \mathbb{E}\mu_{N_1+1} \biggr\} +  \\
	 \label{eq:minklmax3}
	 \sup_{\substack{N_2 \in \mathbb{N} \\ y_1, \dots, y_{N_2} \in \setR}} & \biggl\{ 
	 R_1(y_1 - \epsilon) + \sum_{n=2}^{N_2} R_n(y_n + \epsilon \sigma_n) - R_{n}(y_{n-1} + \epsilon \sigma_n) - R_{N_2+1}(y_{N_2} - \epsilon) - \mathbb{E}\mu_{1} \biggr\} \leq 2\epsilon.
	\end{align}
\end{cor}

For $\epsilon=0$, condition~\eqref{minklmaxwinf} is equivalent to the sequence of
call functions $(R_n)$ being increasing, see Proposition~\ref{prop:consistent}.
For $\epsilon>0$, analogously to the proof of Proposition~\ref{prop:consistent}, we see
that~\eqref{minklmaxwinf} implies
\begin{equation} \label{eq: Rn}
  R_n(x+\epsilon) \leq R_{n+1}(x-\epsilon), \quad x\in\setR, n\in\setN.
\end{equation}
It is clear that~\eqref{eq: Rn} is necessary for the existence of the peacock
$(\nu_n)_{n \in \setN}$, since, by~\eqref{eq:R est} and Proposition~\ref{RFprop}~\ref{convexorderR},
\[
  R_n(x+\epsilon) \leq R_{\nu_n}(x) \leq R_{\nu_{n+1}}(x) \leq R_{n+1}(x-\epsilon), 
    \quad x\in\setR, n\in\setN.
\]
On the other hand, it is easy to show that~\eqref{eq: Rn} is not sufficient for~\eqref{minklmaxwinf}:
\begin{exa}\label{ex:not suff}
   Fix $m>1$ and $\epsilon=1$ and define two measures
   \[
     \mu_1=  \frac 2{m+1}\delta_0+ \frac {m-1}{m+1}\delta_{m+1}, \quad \mu_2=\delta_{m+1},
   \]
   where $\delta$ denotes the Dirac delta. It is simple to check that~\eqref{eq: Rn} is satisfied, i.e.
   \[
   R_{\mu_1}(x+\epsilon) \leq R_{\mu_2}(x-\epsilon), \quad x \in \setR.
   \]
   Now assume that we want to construct a peacock $(\nu_n)_{n=1,2}$ such that $W^\infty(\mu_n,\nu_n)\leq 1$.
   Then the only possible mean for this peacock is $m$, since $\mathbb{E}\mu_1=m-1$ and $\mathbb{E}\mu_2=m+1$ (see the remark before Proposition~\ref{LSwinf}).
   Therefore the peacock has to satisfy $\nu_n \in B^\infty(\mu_n,1)\cap\setM_m$,
   $n=1,2$, and the only possible choice is
   \[
     \nu_1=  \frac 2{m+1}\delta_1+ \frac {m-1}{m+1}\delta_{m+2}, \quad \nu_2=\delta_{m}.
   \]
   But since  $R_{\nu_1}(x) > R_{\nu_2}(x)$ for $x \in (1,m+2)$, $(\nu_n)_{n=1,2}$ is not a peacock; see Figure~\ref{fig:counterex}. 
\begin{figure}[ht]
   \begin{center}
      \includegraphics[width=10cm]{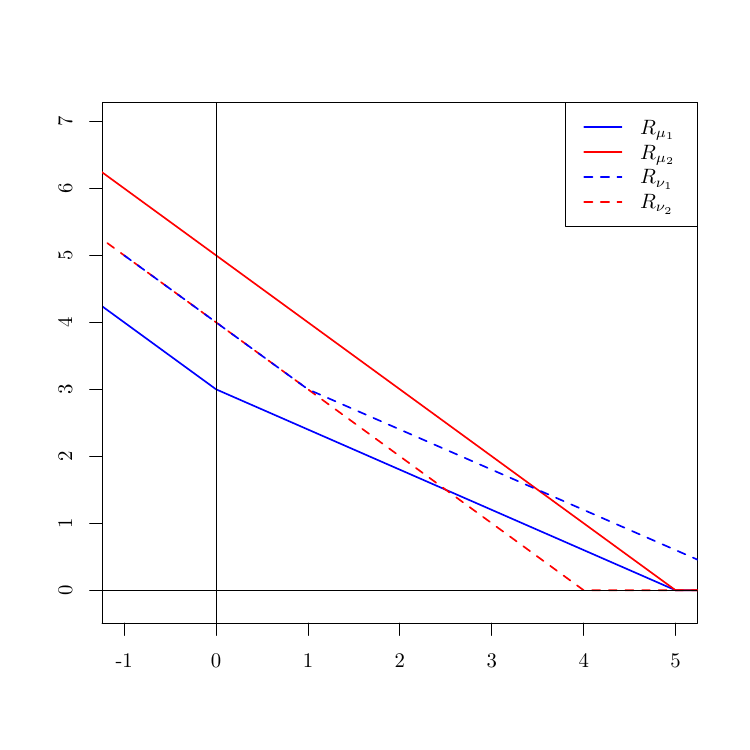}
   \caption{The call functions of $\mu_1$ (lower solid curve) and $\mu_2$ (upper solid curve)
   from Example~\ref{ex:not suff}, for $m=4$ and $\epsilon=1$. 
						The call function of $\nu_1$ is the call function of $\mu_1$ shifted to the right by one. 
						Similarly, shifting the call function of $\mu_2$ by one to the left yields the call function of $\nu_2$. }
     \label{fig:counterex}
   \end{center}
\end{figure}
\end{exa}

If the sequence $(\mu_n)_{n=1,2}$ has just two elements, then it suffices to require~\eqref{minklmaxwinf}
only for $N=1$. It then simply states that there is an $m\in I$ such that
$R_1^{\min}(x;m,\epsilon) \leq R_2^{\max}(x;m,\epsilon)$ for all~$x$, which is clearly
necessary and sufficient for the existence of $(\nu_n)_{n=1,2}$.

%TODO is it never unique?
\begin{exa}
  Unsurprisingly, the peacock from Theorem~\ref{strassenwinf} is in general not unique:
  Let $\epsilon>0$ and consider the constant sequences $R_n(x)=(-x)^+$, $n\in\setN$, and
  \[
    P_n(x,c) =
    \begin{cases}
      -x, & x\leq -\epsilon, \\
      \epsilon - \frac{\epsilon(x+\epsilon)}{c+\epsilon},
        & -\epsilon \leq x \leq c, \\
      0, & x\geq c.
    \end{cases}
  \]
Then, for any $c\in[0,\epsilon]$, it is easy to verify that
the sequence of call functions $P_n(\cdot,c)$ defines a peacock satisfying~\eqref{eq:Winf}.
\end{exa}

\section{Proof and ramifications of Theorem~\ref{strassenwinf}}\label{se:pr ram}

The following theorem furnishes the main step for the induction proof of Theorem~\ref{strassenwinf},
given at the end of the present section. In each induction step,
the next element of the desired peacock should be contained in a certain ball, it should be
larger in convex order than the previous element ($\nu$ in Theorem~\ref{minconstruct}),
and it should be as small as possible
in order not to hamper the existence of the subsequent elements. This leads us to
search for a least element of the set~$A_\mu^\nu$ defined in~\eqref{eq:A}. The conditions defining this least
element translate into inequalities on the corresponding call function.
Part~\ref{ineq} of Theorem~\ref{minconstruct} states that, at each point of the real line,
at least one of the latter conditions becomes an equality.
\begin{thm} \label{minconstruct}
   Let $\mu,\nu$ be two measures in $\setM$ such that the set
   \begin{equation} \label{eq:A}
     A_\mu^\nu:= \Bigl\{ \theta \in B^{\infty}(\mu; \epsilon): \ \nu \leqc \theta \Bigr\}
   \end{equation}
   is not empty.
   \begin{enumerate}[label=(\roman*)]
    \item\label{it:min}
      The set $A_\mu^\nu$ contains a least element $S_{\nu}(\mu)$ with respect to $\leqc$, i.e.\ for every $\theta \in A_\mu^\nu$ we have
      \[
        \nu \leqc S_{\nu}(\mu) \leqc \theta.
      \] 
      Equivalently, if 
      \begin{equation*} \label{callcond}
        R_\nu(x) \leq R_{T(\mu;\mathbb{E}\nu,\epsilon)}(x), \quad x \in \setR,
      \end{equation*}
%NEU
      where~$T(\mu)$ was defined in~\eqref{eq:Rmax}, 
      there exists a pointwise smallest call function $R^*$ which is greater than $R_\nu$ and 
      satisfies \\ $(R^*)'(x) \in [R_\mu'(x-\epsilon),R_\mu'(x+\epsilon)]$ for all $x \in \setR$.
    \item\label{ineq} The call function $R^*$ is a solution of the following variational type inequality:
       \begin{equation}\label{vartype}
        \min\Bigr\{R^*(x)-R_\nu(x), (R^*)'(x)-R'_\mu(x-\epsilon),R'_\mu(x+\epsilon)-(R^*)'(x) \Bigl\}=0, \quad x \in \setR.
      \end{equation}
   \end{enumerate}
\end{thm}

\begin{proof}
   The equivalence in~\ref{it:min} follows from Proposition~\ref{RFprop}~\ref{convexorderR};
   note that the existence of $T(\mu;\mathbb E \nu,\epsilon)$ follows from $A_\mu^\nu\neq \emptyset$.
   We now argue that $S_{\nu}(\mu)$ exists. 
   An easy application of Zorn's lemma shows that there exist minimal elements in $A_\mu^\nu$. 
   If $\theta_1$ and $\theta_2$ are two minimal elements of  $A_\mu^\nu$ then, 
   according to Lemma~\ref{mininkugel}, the measure $\theta_1 \wedge \theta_2$ lies in $B^\infty(\nu, \epsilon) \cap \setM_{\mathbb{E}\nu}$.
   Moreover, the convex function $R_\nu$ nowhere exceeds $R_{\theta_1}$ and $R_{\theta_2}$, and hence we have  
   $R_\nu \leq \mathrm{conv}(R_{\theta_1} \wedge R_{\theta_2})=R_{\theta_1 \wedge \theta_2}.$ 
   Therefore $\theta_1 \wedge \theta_2$ lies in $A_\mu^\nu$. Now clearly $\theta_1 \wedge \theta_2 \leqc \theta_1$ and $\theta_1 \wedge \theta_2 \leqc \theta_2$, 
   and from the minimality we can conclude that $\theta_1 \wedge \theta_2=\theta_1=\theta_2$. 

   Now let $\theta^*$ be the unique minimal element and let $\theta \in A_{\mu}^{\nu}$ be arbitrary.
   Exactly as before we can show that $\theta^* \wedge \theta$ lies in $ A_{\mu}^{\nu}$. 
   Moreover $\theta^*=\theta^* \wedge \theta \leqc \theta$ and therefore $\theta^*$ is the least element of $A_{\mu}^{\nu}$.

   It remains to show~\ref{ineq}. We set 
   \begin{equation}\label{infrstar}
     R^*(x)= \inf\big\{R_\theta(x): \ \theta \in  A_\mu^\nu \big\}.
   \end{equation}
Clearly $R^*$ is a decreasing function with $\lim_{x \rightarrow \infty}R^*(x)=0$ and $\lim_{x \rightarrow -\infty}R^*(x)+x=\mathbb{E}\nu$. We will show that $R^*$ is convex, which is equivalent to the convexity of the epigraph $\mathcal{E}$ of $R^*$. Pick two points $(x_1,y_1), (x_2,y_2) \in \mathcal{E}$. Then there exist measures $\theta_1, \theta_2 \in A_\mu^\nu$ such that $R_{\theta_1}(x_1) \leq y_1$ and $R_{\theta_2}(x_2) \leq y_2$. Using Lemma \ref{mininkugel} once more, we get that  $\theta:= \theta_1 \wedge \theta_2 \in A_\mu^\nu$ and $R_\theta(x_i) \leq y_i, i=1,2$. Therefore, the whole segment with endpoints $(x_1,y_1)$ and $(x_2,y_2)$ lies in the epigraph of $R_\theta$ and hence in $\mathcal{E}$. This implies that $R^*$ is a call function, and as we already know that $A_\mu^\nu$ has a least element $S_\nu(\mu)$,
the measure associated to~$R^*$ has to be $S_\nu(\mu)$. Also, we can therefore conclude that the infimum in \eqref{infrstar} is attained for all $x$.

Now assume that~\eqref{vartype} is wrong. Since all functions appearing in~\eqref{vartype}
are right-continuous,
there must then exist an open interval $(a,b)$ where \eqref{vartype} does not hold, i.e. $R^*(x)>R_\nu(x)$ and $(R^*)'(x) \in (R'_\mu(x-\epsilon), R'_\mu(x+\epsilon) )$ for all $x \in (a,b)$.

\textbf{Case 1:} There exists an open interval $I \subseteq (a,b)$ where $R^*$ is strictly convex.
Then we can pick $x_1 \in I$ and $h_1>0$ such that $x_1+h_1 \in I$ and such that the tangent 
\[
  P_1(x):= R^*(x_1)+(R^*)'(x_1)(x-x_1), \quad x\in [x_1,x_1+h_1]
\]
satisfies $R_\nu(x) < P_1(x)< R^*(x)$ for $x \in (x_1, x_1+h_1]$. Also, since $(R^*)'(x_1)>R'_\mu(x_1-\epsilon)$ and since $R'_\mu$ is right-continuous, we can choose $h_1$ small enough to guarantee $(R^*)'(x_1) \geq R'_\mu(x_1+h_1-\epsilon)$. 
Next pick $x_2 \in (x_1,x_1+h_1)$, such that $R_\mu'(\cdot+\epsilon)$ is continuous at $x_2$ and set 
$$
P_2(x):= R^*(x_2)+(R^*)'(x_2)(x-x_2), \quad x\in [x_2-h_2,x_2].
$$
We can choose $h_2$ small enough to ensure that $R_\nu(x) < P_2(x)< R^*(x)$ and $(R^*)'(x_2) \leq R'_\mu(x_2-h_2+\epsilon)$. Also, if $x_1$ and $x_2$ are close enough together, then there is an intersection of $P_1$ and $P_2$ in $(x_1,x_2)$. Now the function
$$
\widetilde{R}(x):= \begin{cases} 
P_1(x)\vee P_2(x), & x\in [x_1,x_2],  \\
R^*(x), & \mbox{otherwise},
\end{cases}
$$
is a call function which is strictly smaller than $R^*$ and satisfies  $\widetilde{R}'(x) \in [R_\mu'(x-\epsilon),R_\mu'(x+\epsilon)]$ for all $x \in \setR$. This is a contradiction to \eqref{infrstar}.
See Figure~\ref{fig:proof36a} for an illustration.
\begin{figure}[ht]
   \begin{center}
      \includegraphics[width=10cm]{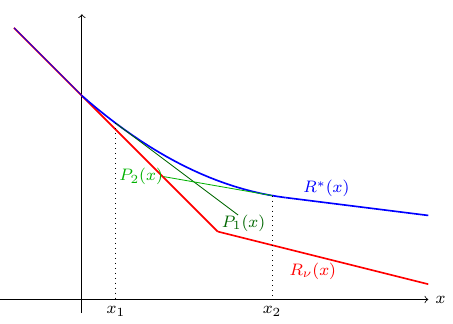}
   \caption{\label{fig:proof36a}
     Case~1 of the proof of Theorem~\ref{minconstruct}. If~$R^*$ is strictly convex, then we
     can deform it using two appropriate tangents, contradicting minimality of the associated measure. 
   }
   \end{center}
\end{figure}

\textbf{Case 2:} If there is no open interval in $(a,b)$ where $R^*$ is strictly convex, then $R^*$ has to be affine on some closed interval $I \subseteq (a,b)$ (see p.\ 7 in \cite{Ro73}). Therefore, there exist $k,d$ in $\setR$ such that 
$$
R^*(x)=kx+d, \quad x \in I. 
$$ 
By Proposition~\ref{RFprop}~\ref{limitcall}, the slope $k$ has to lie in the open interval $(-1,0)$, since $R^*$ is greater than $R_\nu$ on $I$. We set 
\begin{align*}
  a_1 &:= \sup\bigl\{x \in \setR: (R^*)'(x) < k\bigr\}>-\infty, \\
  b_1 &:= \inf\bigl\{x \in \setR: (R^*)'(x) >k\bigr\}<\infty;
\end{align*}
the finiteness of these quantities follows from Proposition~\ref{RFprop}~\ref{limitcall}.
{}From the convexity of $R_\nu$ and the fact that $R_\nu \leq R^*$, we get that $R^*(x)>R_\nu(x)$ for all $x \in (a_1, b_1)$, as well as  $(R^*)'(x)>R'_\mu(x-\epsilon)$ for all $x \in (a_1,b)$  and  $(R^*)'(x)<R'_\mu(x+\epsilon)$ for all $x \in (a,b_1)$. We now define lines~$P_1$ and~$P_2$, with analogous roles
as in Case~1. Their definitions depend on the behavior of~$(R^*)'$ at~$a_1$ and~$b_1$. 

If $(R^*)'(a_1-)<k$, then we set $x_1=a_1$ and $P_1(x)=R^*(x_1)+k_1(x-x_1)$ for $x\geq x_1$,
with an arbitrary $k_1 \in ((R^*)'(x_1-),k)$; see Figure~\ref{fig:proof36b}.

If, on the other hand,  $(R^*)'(a_1-)=k$, then we can find $x_1<a_1$ such that $R^*(x_1)>R_\nu(x_1)$ and $(R^*)'(x_1)>R'_\mu(x_1-\epsilon)$. In this case we define
\begin{equation*}
P_1(x):= R^*(x_1)+(R^*)'(x_1)(x-x_1), \quad x \geq x_1. 
\end{equation*}

Similarly, if $(R^*)'(b_1)>k$, then we define $x_2=b_1$ and $P_2(x)=R^*(x_2)+k_2(x-x_2)$ for $x\leq x_2$ and for $k_2 \in (k,(R^*)'(b_1))$, and otherwise we can find $x_2>b_1$ such that $R^*(x_2)>R_\nu(x_2)$ and  $(R^*)'(x_2)<R'_\mu(x_2+\epsilon)$. We then set
\begin{equation*}
P_2(x):= R^*(x_2)+(R^*)'(x_2)(x-x_2), \quad x \leq x_2. 
\end{equation*}
We can choose $h_1,h_2>0$, $\widetilde{d}<d$ and $k_1,k_2$ such that the function
$$
\widetilde{R}(x):= \begin{cases} 
P_1(x), & x\in [x_1,x_1+h_1],  \\
kx+\widetilde{d}, & x\in [x_1+h_1, x_2-h_2] \\
P_2(x), & x\in [x_2-h_2,x_2], \\
R^*(x), & \mbox{otherwise},
\end{cases}
$$
is a call function which is strictly smaller that $R^*$ but not smaller than $R_\nu$. Also, if $h_1$ and $h_2$ are small enough we have $\widetilde{R}'(x) \in [R_\mu'(x-\epsilon), R_\mu'(x+\epsilon)]$ for all $x \in \setR$, which is a contradiction to~\eqref{infrstar}.
\end{proof}

\begin{figure}[ht]
   \begin{center}
      \includegraphics[width=10cm]{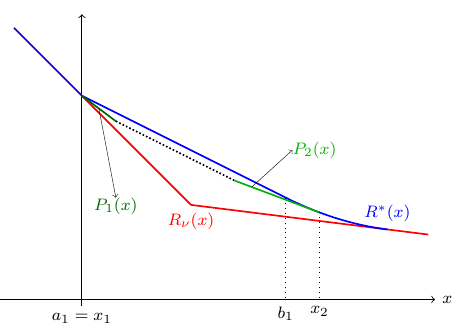}
   \caption{\label{fig:proof36b}
     Case~2 of the proof of Theorem~\ref{minconstruct}, with
      $(R^*)'(a_1-)<k$ and $(R^*)'(b_1)=k$.}
     \label{fig:tilde R}
   \end{center}
\end{figure}

In part~\ref{it:min} of Theorem~\ref{minconstruct}, we showed that $A_\mu^\nu$ has a least element.
The weaker statement that it has an infimum follows from~\cite{KeRo00}, p.~162;
there it is shown that \emph{any} subset of the lattice  $(\setM_m, \leqc)$  has an infimum.
(The stated requirement that the set be bounded from below is always satisfied,
as the Dirac delta~$\delta_m$ is the least
element of $(\setM_m, \leqc)$.)
This infimum is, of course, given by the least element $S_{\nu}(\mu)$ that we found.

%
%\begin{rem}
   If $\nu=\delta_m$, then $S_{\nu}(\mu) = S(\mu)$, the least element
   from Proposition~\ref{LSwinf}.  In this case we have 
   \[
     (R^*)'(x)=\begin{cases}
     R'_\mu(x-\epsilon),&  x< x^*, \\
     R'_\mu(x+\epsilon),&  x\geq  x^*, 
     \end{cases}
   \]
   where $x^*$ is the unique solution of 
   \[
     m+R_\mu(x-\epsilon)-\bigl(\mathbb{E}\mu+\epsilon \bigr)  = R_\mu(x+\epsilon).
   \] 
%\end{rem}

The following corollary establishes an alternative representation of the inequality~\eqref{vartype},
which we will use to prove Theorem~\ref{strassenwinf}.
Note that, in general, \eqref{vartype} has more than one solution, not all of which are call functions. 
   However, $R^*$ is always a solution.
\begin{cor} \label{xycor}
   Assume that the conditions from Theorem~\ref{minconstruct} hold and denote the call function of $S_\nu(\mu)$
   by $R^*$. Then for all $x \in \setR$ there exists $y \in \setR \cup \{\pm \infty\}$ such that
   \[
     R^*(x)=R_\nu(y)-R_\mu(y+\epsilon\sigma)+R_\mu(x+\epsilon\sigma),
   \]
   where $\sigma=\sgn(y-x)$. Here and in the following we set $R(\infty)=0$ for all call functions $R$ and 
   \[
     R_1(-\infty \pm \epsilon)-R_{2}(-\infty \pm \epsilon):= \lim_{x\rightarrow -\infty} (R_1(x \pm \epsilon)-R_{2}(x \pm \epsilon)),
   \]
   for call functions $R_1$ and $R_2$.
\end{cor}

\begin{proof}
   By Theorem~\ref{minconstruct} we know that $R^*$ is a solution of~\eqref{vartype}.
   Let~$x$ be an arbitrary real number.
   If $R^*(x)=R_\nu(x)$, then the above relation clearly holds for $y=x$. Otherwise,
   we have $R^*(x)>R_\nu(x)$, and one of the other two expressions on the left hand
   side of~\eqref{vartype} must vanish at~$x$.
   First we assume that $(R^*)'(x)=R_\mu'(x+\epsilon)$. Define
   \[
     y:= \inf\{z\geq x: \ (R^*)'(z)<R_\mu'(z+\epsilon)\}.
   \]
   If $y<\infty$, then by definition $(R^*)'(y)<R_\mu'(y+\epsilon)$.
   By~\eqref{vartype}, we have $R^*(y)=R_\nu(y)$. 
   It follows that
   %NEU
   \begin{align*}
     R^*(z) =& R^*(y) + \int\nolimits_{y}^z (R^*)'(x) \; dx =  R_\nu(y) + \int\nolimits_{y}^z R_\mu'(x+\epsilon) \; dx = \\
     &R_\nu(y)-R_\mu(y+\epsilon)+R_\mu(z+\epsilon), \quad  \text{for all} \ z \in [x,y].
   \end{align*}
   If $y=\infty$, then this equation, i.e.\ $R^*(z)=R_\mu(z+\epsilon)$, $z\geq x$, also holds.
   
   If, on the other hand, $(R^*)'(x)=R_\mu'(x-\epsilon)$, 
   then we similarly define $y:= \sup\{z\leq x: \ (R^*)'(z)>R_\mu'(z-\epsilon)\}$. 
   If $y>-\infty$ then $(R^*)'(y-)>R_\mu'((y-\epsilon)-)$ and hence $R^*(y)=R_\nu(y)$ by~\eqref{vartype}. Therefore we can write
   \[
     R^*(z)=R_\nu(y)-R_\mu(y-\epsilon)+R_\mu(z-\epsilon), \quad  \text{for all} \ z \in [y,x].
   \]
   If $y=-\infty$ then $(R^*)'(z)=R_\mu'(z-\epsilon)$ for all $z \leq x$. The above equation holds if we take the limit $y\rightarrow \infty$ on the right hand side.
\end{proof}

\begin{cor} \label{Winfcor}
   Using Proposition~\ref{LSwinf} and Theorem~\ref{minconstruct}, for a
   given sequence of measures $(\mu_n)_{n \in \setN}$ in~$\setM$, we inductively define the measures
   \[
     \theta_1 = S(\mu_1;m,\epsilon), \quad \theta_k = S_{\theta_{k-1}}(\mu_k), \quad k \geq 2,
   \]
   if the sets 
   \[
     \big\{\nu \in B^\infty(\mu_{k},\epsilon): \ \theta_{k-1} \leqc \nu \big\}
   \]
   are not empty. Then the following relation holds:
   \[
     R_{\theta_{n}}(x)=R_{\theta_{n-1}}(y)-R_{\mu_{n}}(y+\epsilon \sigma)+R_{\mu_{n}}(x+\epsilon \sigma),
   \]
   where $n \geq 2$, $y\in \setR\cup\{\pm \infty \}$ depends on $x$ and $\sigma=\sgn(y-x)$.
\end{cor}

\begin{proof}
   The result follows by simply applying Theorem~\ref{minconstruct} and Corollary~\ref{xycor}  with $\nu=\theta_{n-1}$ and $\mu=\mu_{n}$.
\end{proof}

The next corollary will be useful later on in Theorem~\ref{thm:strassenW} and is an easy consequence of~\eqref{vartype}.

\begin{cor}\label{cor:munufinite}
	Let $\mu$, $\nu$ be as in Theorem~\ref{minconstruct} and additionally assume that both measures
	have finite support. Then $S_\nu(\mu)$ has finite support too.
\end{cor}

\begin{proof}
By \ref{RFrel} of Proposition~\ref{RFprop}, the finiteness of the support of a measure~$\theta$ is equivalent to $R_\theta'$ having a finite range. Therefore, we can partition the real line into a finite number of intervals $I_1, \dots, I_N$ such that for all $n \in \{1, \dots N\}$
the functions $R_\nu', R_\mu'(.-\epsilon)$ and $R_\mu'(.+\epsilon)$ are constant on~$I_n$. Since $R_{S_\nu(\mu)}$ solves~\eqref{vartype}, we can conclude that $R'_{S_\nu(\mu)}$ takes at most three distinct values on each $I_n$. Hence,  $R'_{S_\nu(\mu)}$ is piecewise constant and $S_\nu(\mu)$ has finite support.
\end{proof}

We can now prove Theorem~\ref{strassenwinf}, our main result on approximation
by peacocks. We first prove the ``if'' direction, which, unsurprisingly,
is the more difficult one.
\begin{proof}[Proof of Theorem~\ref{strassenwinf}]
   Suppose that~\eqref{minklmaxwinf} holds for some $m\in I$ and all $N\in\setN$,
   $x_1,\dots,x_N\in\setR$. We will inductively construct a sequence  $(P_n)_{n \in \setN}$ of call functions, 
   which will correspond to the measures $(\nu_n)_{n \in \setN}$.
   Define $P_1= R_1^{\min}(\: . \:;m,\epsilon)$. 
   For $N=1$,~\eqref{minklmaxwinf} guarantees that $R_1^{\min}(x) \leq R_2^{\max}(x)$. 
   Note that the continuity of the $R_n$ guarantee that~\eqref{minklmaxwinf} also holds for $x_n \in \{\pm \infty\}$, if we set $\sgn(\infty-\infty)=\sgn(-\infty+\infty)=0$.
   We can now use Theorem~\ref{minconstruct} together with Corollary~\ref{xycor}, 
   with $R_\nu=R_1^{\min}$ and $R_\mu=R_2$, to construct a call function $P_2$, which satisfies
   \[
     P_2(x)=R_1^{\min}(x_1)+R_2(x+\epsilon\sigma)-R_2(x_1+\epsilon\sigma), \quad x \in \setR,
   \]
   where $\sigma = \sgn(x_{1}-x)$, and $x_1$ depends on $x$.
   If we use \eqref{minklmaxwinf} we get that
   \[
     R_1^{\min}(x_1)+R_2\bigl(x+\epsilon \sigma_2 \bigr)-R_2\bigl(x_{1}+\epsilon \sigma_2 \bigr) \leq R_{n}^{\max}(x; m, \epsilon), \quad n \geq 3, \ x_1,x \in \setR.%\cup\{\pm \infty\}.
   \]
   Hence $P_2(x) \leq  R_n^{\max}(x)$ for all $x \in \setR$ and for all $n \geq 3$.  
   Now suppose that we have already constructed a finite sequence $(P_1, \dots, P_N)$ such that $P_n \leq P_{n+1}$, $1\leq n <N$, 
   and such that $P_N \leq  R_n^{\max}$ for all $x \in \setR$ and for all $n \geq N+1$. 
   Then by induction we know that for all $x\in \setR$ there exists $(x_1, \dots, x_{N-1})$ such that
   \[
     P_N(x)=R_1^{\min}(x_1) + \sum_{n=2}^{N-1}\Big(R_n\bigl(x_n+\epsilon \sigma_n \bigr)-R_n\bigl(x_{n-1}+\epsilon \sigma_n \bigr)\Big) + R_N\bigl(x+\epsilon \sigma_N \bigr)-R_N\bigl(x_{N-1}+\epsilon \sigma_N \bigr),
   \]
   with $\sigma_N= \sgn(x_{N-1}-x)$. In particular, we have $P_N\leq R_{N+1}^{\max}$. 
   We can therefore again use Corollary~\ref{xycor}, with $R_\mu=R_{N+1}$ and $R_\nu=P_{N}$, 
   to construct a call function $P_{N+1}$, such that 
   \begin{multline*}
     P_{N+1}(x)=R_1^{\min}(x_1) + \sum_{n=2}^{N}\Big(R_n\bigl(x_n+\epsilon \sigma_n \bigr)-R_n\bigl(x_{n-1}+\epsilon \sigma_n \bigr)\Big)\\
      + R_{N+1}\bigl(x+\epsilon \sigma_{N+1} \bigr)-R_{N+1}\bigl(x_{N}+\epsilon \sigma_{N+1}\bigr),
   \end{multline*}
   where $\sigma_{N+1}= \sgn(x_{N}-x)$ and $(x_1, \dots, x_N)$ depend on $x$. 
   Assumption~\eqref{minklmaxwinf} guarantees that $P_{N+1} \leq R_n^{\max}$ for all $n \geq N+1$. 

   We have now constructed a sequence of call functions, such that $P_n \leq P_{n+1}$. Their associated measures, 
   which we will denote by $\nu_n$, satisfy $W^\infty(\mu_n, \nu_n)\leq \epsilon$ and $\nu_n \leqc \nu_{n+1}$. 
   Thus we have constructed a peacock with mean $m$.
   
   We proceed to the proof of the (easier) ``only if'' direction
   of Theorem~\ref{strassenwinf}. Thus,
   assume that $(\nu_n)_{n \in \setN}$ is a peacock such that $W^\infty(\mu_n, \nu_n) \leq \epsilon$ and set $m=\mathbb{E}\nu_1$. %Then it follows that $m \in I$ and hence $I$ cannot be empty. 
   Denote the call function of $\nu_n$ by $P_n$. We will show by induction that~\eqref{minklmaxwinf} holds. For $N=1$ we have 
   \[
     R_1^{\min}(x;m,\epsilon) \leq P_1(x) \leq P_2(x) 
     \leq R_2^{\max}(x;m,\epsilon), \quad x\in \setR,
     %\leq \Bigl(s+R_{2}(x+\epsilon)-\bigl(\mathbb{E}\mu_{2}-\epsilon \bigr) \Bigr)\wedge R_{2}(x-\epsilon),
   \]
   by Proposition~\ref{LSwinf}. 

   For $N=2$ and $x_1 \leq x_2$ we have
   \begin{align*} 
     R_1^{\min}(x_1;m,\epsilon) + R_2(x_2-\epsilon)-R_2(x_1-\epsilon)
     &\leq P_2(x_1) + \int\nolimits_{x_1}^{x_2} R_2'(z-\epsilon) \, dz  \\
     &\leq P_2(x_1) + \int\nolimits_{x_1}^{x_2} P_2'(z) \, dz \\
     &=P_2(x_2) \leq  P_3(x_2) \leq R_3^{\max}(x_2;m,\epsilon).
     %\leq\Bigl(s+R_{3}(x_2+\epsilon)-\bigl(\mathbb{E}\mu_{3}-\epsilon \bigr) \Bigr)\wedge R_{2}(x_2-\epsilon)
   \end{align*} 
   Similarly, if $x_2 \leq x_1$,
   \begin{align*} 
     R_1^{\min}(x_1;m,\epsilon) + R_2(x_2+\epsilon)-R_2(x_1+\epsilon) &\leq P_2(x_1) - \int\nolimits_{x_2}^{x_1} R_2'(z+\epsilon) \, dz  \\
     &\leq P_2(x_1) - \int\nolimits_{x_2}^{x_1} P_2'(z) \, dz \\
     &=P_2(x_2) \leq  P_3(x_2) \leq R_3^{\max}(x_2;m,\epsilon).
%\leq\Bigl(s+R_{3}(x_2+\epsilon)-\bigl(\mathbb{E}\mu_{3}-\epsilon \bigr) \Bigr)\wedge R_{2}(x_2-\epsilon)
   \end{align*} 
   If \eqref{minklmaxwinf} holds for $N-1$ and $x_{N-1}\leq x_N$, then
   \begin{align*}
     R_1^{\min}(x_1; m, \epsilon)& + \sum_{n=2}^{N}\left(R_n\bigl(x_n+\epsilon \sigma_n \bigr)-R_n\bigl(x_{n-1}+\epsilon \sigma_n \bigr)\right) \\
     &\leq P_{N-1}(x_{N-1})+R_N\bigl(x_N-\epsilon \bigr)-R_N\bigl(x_{N-1}-\epsilon \bigr) \\
     &\leq P_N(x_{N-1}) + \int\nolimits_{x_{N-1}}^{x_N} P_N'(z) \, dz \\
     &\leq P_{N+1}(x_N) \leq R_{N+1}^{\max}(x_N;m,\epsilon).
     %&\leq \Bigl(s+R_{N+1}(x_N+\epsilon)-\bigl(\mathbb{E}\mu_{N+1}-\epsilon \bigr) \Bigr)\wedge R_{N+1}(x_N-\epsilon).  
   \end{align*}
   The case where $x_{N-1}\geq x_N$ can be dealt with similarly.
\end{proof}

\begin{proof}[Proof of Corollary~\ref{strassenwinf2}]
	First, by going through the proof of Theorem~\ref{strassenwinf} a second time, we see that $R_{N+1}^{\max}(x_N; m, \epsilon)$, in the definition of~$\Phi_N$ can be replaced by 
	\[
	   \widetilde{R_{N+1}^{\max}}(x_N; m, \epsilon):= \Bigl(m+R_{N+1}(x_N +\epsilon)-\bigl(\mathbb{E}\mu-\epsilon \bigr)\Bigr) \wedge R_{N+1}(x_N -\epsilon),
	\]
	which is $R_{N+1}^{\max}(x_N; m, \epsilon)$ without the convex envelope. 
	
	Next, we can split up~\eqref{minklmaxwinf} into four inequalities according to the different components of  $\widetilde{R_{N+1}^{\max}}$ and $R_{1}^{\min}$. In two of these inequalities $m$ does not appear, and these are exactly equations~\eqref{eq:minklmax1} and \eqref{eq:minklmax2}. The remaining two inequalities are given by
	\begin{align*}
	  R_1(x_1 - \epsilon) + m - (\mathbb{E}\mu_1 + \epsilon) + 
	  \sum_{n=2}^{N}R_n\bigl(x_n+\epsilon \sigma_n \bigr)-R_n\bigl(x_{n-1}+\epsilon \sigma_n \bigr)-
	  R_{N+1}(x_N + \epsilon) \leq 0, \\
	    R_1(x_1 + \epsilon)  +
	  \sum_{n=2}^{N}R_n\bigl(x_n+\epsilon \sigma_n \bigr)-R_n\bigl(x_{n-1}+\epsilon \sigma_n \bigr)-
	  \Bigl(R_{N+1}(x_N + \epsilon)+ m - (\mathbb{E}\mu_{N+1} - \epsilon)\Bigr) \leq 0.
    \end{align*}
    In particular, $m$ can only exist if  
    \begin{align}
    \nonumber
      \sup_{\substack{N_1 \in \mathbb{N} \\ x_1, \dots, x_{N_1} \in \setR}} & \biggl\{ 
        R_1(x_1 + \epsilon) + \sum_{n=2}^{N_1} R_n(x_n + \epsilon \sigma_n) - R_{n}(x_{n-1} + \epsilon \sigma_n) - R_{N_1+1}(x_{N_1} + \epsilon) + \mathbb{E}\mu_{N_1+1} \biggr\} - \epsilon \leq \\
        \label{eq:mbounds}
      \inf_{\substack{N_2 \in \mathbb{N} \\ y_1, \dots, y_{N_2} \in \setR}} & \biggl\{ 
        -R_1(y_1 - \epsilon) - \sum_{n=2}^{N_2} R_n(y_n + \epsilon \sigma_n) - R_{n}(y_{n-1} + \epsilon \sigma_n) + R_{N_2+1}(y_{N_2} - \epsilon) + \mathbb{E}\mu_{1} \biggr\} + \epsilon,
    \end{align}
    in which case~$m$ can be chosen arbitrarily from the closed interval with bounds given by the left hand side resp.\ right hand side of~\eqref{eq:mbounds}. A simple modification of~\eqref{eq:mbounds} yields~\eqref{eq:minklmax3}.
\end{proof}

\begin{rem}
   In Theorem~\ref{strassenwinf}, it is actually not necessary that the balls centered at
   the measures~$\mu_n$ are all of the  same size. The theorem easily generalizes to the following result: 
   For $m \in \setR$, a sequence of non-negative numbers $(\epsilon_n)_{n \in \setN}$, and a sequence of measures $(\mu_n)_{n \in \setN}$ in $\setM$, define
   \begin{multline}\label{eq:Phi-gen}
     \Phi_N(x_1,\dots,x_N;m,\epsilon_1,\dots,\epsilon_{N+1})=
     R_1^{\min}(x_1; m, \epsilon_1) \\
     + \sum_{n=2}^{N}\Big(R_n(x_n+\epsilon_n \sigma_n )-R_n(x_{n-1}+\epsilon_n \sigma_n )\Big)
     - R_{N+1}^{\max}(x_N; m, \epsilon_{N+1}),\\
     N \in \setN,\ x_1,\dots,x_N \in \setR,
   \end{multline}
   with~$\sigma_n$ defined in~\eqref{eq:sigma},
   and assume that
   \[
     I:= \bigcap_{n \in \setN}[\mathbb{E}\mu_n-\epsilon_n, \mathbb{E}\mu_n+\epsilon_n] 
   \]
   is not empty. Then there exists a peacock $(\nu_n)_{n \in \setN}$ 
   such that 
   \[
     W^\infty(\mu_n, \nu_n) \leq \epsilon_n, \quad \mbox{for all} \ n \in \setN,
   \]
   if and only if for some $m \in I$
   and for all $N \in \setN$,  $x_1, \dots, x_N \in \setR$, %\cup\{\pm \infty\}$,#
   we have
   \begin{equation*} 
     \Phi_N(x_1,\dots,x_N;m,\epsilon_1, \dots, \epsilon_{N+1})\leq0.
   \end{equation*}
   To prove this result, simply replace $\epsilon$ by $\epsilon_n$ in the
   proof of Theorem~\ref{strassenwinf}.
\end{rem}

\begin{rem}
   If a probability metric is comparable with the infinity Wasserstein distance,
   then our Theorem~\ref{strassenwinf} implies a corresponding result about that metric
   (but, of course, not an ``if and only if'' condition). 
   %Denote by $d^\mathrm{L}$ the Lévy distance, defined by
   %\[
     %d^\mathrm{L}(\mu,\nu)=\inf \Bigl\{\epsilon >0: R'_\mu(x-\epsilon)-\epsilon
     %\leq R'_\nu(x) \leq R'_\mu(x+\epsilon)+\epsilon,\ \forall x \in \setR \Bigr\}, \quad \mu,\nu \in \setM, 
   %\]
   Denote by $W^p$ the $p$-Wasserstein distance ($p\geq 1$), defined by
   \[
     W^p(\mu,\nu)=\inf \Bigl(\mathbb{E}[|X-Y|^p]\Bigr)^{1/p}, \quad \mu,\nu \in \setM. 
   \]
   The infimum is taken over all probability spaces $(\Omega,\mathcal{F}, \pp)$ and random pairs $(X,Y)$ with marginals given by $\mu$ and $\nu$.
   Clearly, we have that for all $\mu,\nu \in \setM$ and $p\geq 1$
   \[
     %W^\infty(\mu,\nu) \geq \max \Bigl\{d^\mathrm{L}(\mu,\nu), W^p(\mu,\nu)\Bigr\}.
     W^\infty(\mu,\nu) \geq W^p(\mu,\nu).
   \]
   Hence, given a sequence $(\mu_n)_{n \in \setN}$,~\eqref{minklmaxwinf} is a sufficient
   condition 
   for the existence of a peacock $(\nu_n)_{n \in \setN}$, such that
   %$d^\mathrm{L}(\mu_n,\nu_n) \leq \epsilon$ resp.~
   $W^p(\mu_n,\nu_n) \leq \epsilon$ for all $n \in \setN$. 
   But since the balls with respect to $W^p$ % and $d^\mathrm{L}$
   are in general strictly larger 
   than the balls with respect to $W^\infty$, we cannot expect~\eqref{minklmaxwinf} to be necessary.
\end{rem}

\section{Approximation by peacocks: infinity Wasserstein distance (continuous time)}\label{se:cont}
%\section{Strassen's theorem for the infinity Wasserstein distance: continuous time}\label{se:cont}

In this section we will formulate a version of Theorem \ref{strassenwinf} for continuous index sets.
We generalize the definition of $\Phi_N$ from \eqref{eq:Phi} as follows.
For finite sets $\mathcal{T}=\{t_1,\dots,t_{N+1}\} \subseteq [0,1]$ with
$t_1 < t_2 < \dots <t_{N+1}$, we set
\begin{multline}\label{eq:Phigen}
  \Phi_\mathcal{T}(x_1,\dots,x_N;m,\epsilon)=
  R_{t_1}^{\min}(x_1; m, \epsilon) \\
  + \sum_{n=2}^{N}\Big(R_{t_n}(x_n+\epsilon \sigma_n )-R_{t_{n-1}}(x_n+\epsilon \sigma_n )\Big)
   - R_{t_{N+1}}^{\max}(x_N; m, \epsilon).
\end{multline}
Here, $R_{t_1}^{\min}$ is the call function of $S(\mu_{t_1}; m, \epsilon)$, $R_{t_{N+1}}^{\max}$ is the call function of $T(\mu_{t_{N+1}}; m ,\epsilon)$, and \textnormal{$\sigma_n= \sgn(x_{n-1}-x_n)$}
depends on~$x_{n-1}$ and~$x_n$.
Using $\Phi_\mathcal{T}$, we can now formulate a necessary and sufficient condition for
the existence of a peacock within $\epsilon$-distance. The continuity assumption~\eqref{distcont}
occurs in the proof in a natural way; we do not know to which extent it can be relaxed.
%Note that a direct generalization of condition \eqref{minklmaxwinf} for an uncountable set of indices is not possible. 
%We will therefore make a continuity assumption to help us out.

\begin{thm}\label{thm:01}
   Assume that $(\mu_t)_{t \in [0,1]}$ is a family of measures in $\setM$ such that 
   \[
     I:=\bigcap_{t \in [0,1]}[\mathbb{E}\mu_t-\epsilon, \mathbb{E}\mu_t+\epsilon]
   \] 
   is not empty and such that
   \begin{equation} \label{distcont}
     \lim_{s \uparrow t} \mu_s = \mu_t \ \text{weakly}, \quad t \in [0,1].
   \end{equation}
   Then there exists a peacock $(\nu_t)_{t \in [0,1]}$ with 
   \[
     W^\infty\bigl(\mu_t, \nu_t\bigr) \leq \epsilon,  \quad \mbox{for all} \ t \in [0,1],
   \]
   if and only if there exists $m \in I$ such that for all finite sets $ \mathcal{T}=\{t_1,\dots,t_{N+1}\} \subset \setQ\cap[0,1]$ 
   with $t_1 < t_2 < \dots <t_{N+1}$, and for all $x_1, \dots, x_N \in \setR$ we have that
   \begin{equation} \label{winfcontcond}
     \Phi_\mathcal{T}(x_1,\dots,x_N;m,\epsilon) \leq 0.
   \end{equation}
   In this case it is possible to choose $\mathbb{E}\nu_t=m$ for all $t \in [0,1]$. 
\end{thm}

\begin{proof}
   By Theorem \ref{strassenwinf}, condition \eqref{winfcontcond} is clearly necessary for the existence of such a peacock. 
   In order to show that it is sufficient, fix $m \in I$ such that \eqref{winfcontcond} holds. We will first construct $\nu_q$ for $q \in \mathcal D$,
   where
   \[
     \mathcal D = \{ a 2^{-b} \in [0,1]: a,b\in\mathbb{N}_0 \}.
   \] 
  For $n\in\mathbb N$, define measures (recall the notation from Theorem~\ref{minconstruct})
  \[
    \theta_0^{(n)}=\mu_0 \quad \text{and}\quad
    \theta_k^{(n)} = S_{\theta^{(n)}_{k-1}}(\mu_{k2^{1-n}}), \quad 1\leq k\leq 2^{n-1}.
  \]
  Condition~\eqref{winfcontcond} guarantees that these measures exist.
  Obviously,
  \begin{equation}\label{eq:theta inc}
    \theta_k^{(n)} \leqc \theta_{k+1}^{(n)}, \quad n\in\setN,\ 0\leq k<2^{n-1}.
  \end{equation}
  We show by induction on~$k$ that
  \begin{equation}\label{eq:thetak}
    \theta_k^{(n)} \leqc \theta_{2k}^{(n+1)}, \quad n\in\mathbb N,\ 0\leq k\leq 2^{n-1}.
  \end{equation}
  For $k=1$, we have
  \begin{align*}
    \theta_1^{(n)} &= S_{\mu_0}(\mu_{2^{1-n}}) \\
    &\leqc S_{\theta_1^{(n+1)}}(\mu_{2^{1-n}}) = \theta_2^{(n+1)}, \quad n\in\mathbb N.
  \end{align*}
  For $k\geq 1$, we obtain
   \begin{align*}
    \theta_{k+1}^{(n)} &= S_{\theta_{k}^{(n)}}(\mu_{(k+1)2^{1-n}}) \\
    &\leqc S_{\theta_{2k}^{(n+1)}}(\mu_{(k+1)2^{1-n}}) \\
    &\leqc S_{\theta_{2k+1}^{(n+1)}}(\mu_{(k+1)2^{1-n}}) = \theta_{2k+2}^{(n+1)}, \quad n\in\mathbb N,
  \end{align*}
   where the first ``$\leqc$'' follows from the induction hypothesis and the definition
   of $S_{\cdot}(\cdot)$. Thus, \eqref{eq:thetak} is true.
   
   For $q=a2^{-b}\in\mathcal D$, define
   \[
     \nu_q^{(n)} = \theta^{(n)}_{a2^{n-b-1}} \in B^\infty(\mu_q,\epsilon),\quad n>b.
   \]
   By~\eqref{eq:thetak}, we have $\nu_q^{(n)} \leqc \nu_q^{(n+1)}$, $n>b$.
   Let $R_n$ be the call function associated to $\nu_q^{(n)}$.
   Then we have 
   \begin{equation}\label{contproof1}
     R_{S(\mu_q;m,\epsilon)} \leq R_n \leq R_{n+1} \leq R_{T(\mu_q;m,\epsilon)}, \quad n >b,
   \end{equation}
   and thus the bounded and increasing sequence $(R_n)$ converges pointwise to a function $R$. 
   As a limit of decreasing convex functions, $R$ is also decreasing and convex and together with \eqref{contproof1} 
   we see that $R$ is a call function with $\lim_{x \rightarrow -\infty}= R(x)+x=m$. 
   Therefore $R$ can be associated to a measure $\nu_q \in \setM_m$. 

   Next, we will show that $\nu_q \in B^\infty(\mu_q, \epsilon)$. From the convexity of the $R_n$ we get that 
   \begin{align*}
     R'(x) &= \lim_{h \downarrow 0} \lim_{n \rightarrow \infty} \frac {R_n(x+h)-R_n(x)}h \\
     &\geq \lim_{h \downarrow 0} \lim_{n \rightarrow \infty} R_n'(x+h) \\
      &\geq \lim_{h \downarrow 0} \lim_{n \rightarrow     \infty} R_{\mu_q}'(x+h-\epsilon) = R_{\mu_q}'(x-\epsilon),
   \end{align*}
   and similarly
   \begin{align*}
   R'(x) &= \lim_{h \downarrow 0} \lim_{n \rightarrow \infty} \frac {R_n(x+h)-R_n(x)}h\\
    &\leq \lim_{h \downarrow 0} \lim_{n \rightarrow \infty} R_n'(x) \\
     &\leq \lim_{n \rightarrow \infty} 
   R_{\mu_q}'(x+\epsilon) = R_{\mu_q}'(x+\epsilon),
   \end{align*}
   thus $W^\infty(\nu_q,\mu_q)\leq \epsilon$. 

   For two elements $q<q'$ of $\mathcal D$, it is an immediate consequence of~\eqref{eq:theta inc}
   that $\nu_q^{(n)} \leqc \nu_{q'}^{(n)}$ for large~$n$, and therefore
   $\nu_q \leqc \nu_{q'}$. It follows that $(\nu_t)_{t\in\mathcal D}$ is a peacock.
   Now pick $t \in [0,1]\setminus\mathcal D$ and a sequence $\mathcal D\ni q_n\nearrow t$.
   The sequence of call functions corresponding to $\nu_{q_n}$ increases
   and converges to a call function, which is clearly independent of the choice of~$q_n$.
   Denote the associated measure by~$\nu_t$; it satisfies $\mathbb{E}{\nu_t}=m$.  
   Fix $x\in\setR$ and define
   \[
     \mathcal{H} = \{h\in\mathbb R : F_{\mu_t}\ \text{continuous at}\ x+h-\epsilon\}.
   \]
   Note that $(0,\infty)\setminus \mathcal H$ is countable. We obtain
   \begin{align*}
      R'_{\nu_t}(x) &= \lim_{h \downarrow 0} \lim_{n \rightarrow \infty}
       \frac {R_{\nu_{q_n}}(x+h)-R_{\nu_{q_n}}(x)}h \\
       &\geq \lim_{h \downarrow 0} \lim_{n \rightarrow \infty} R_{\nu_{q_n}}'(x+h) \\
       &\geq \lim_{h \downarrow 0} \lim_{n \rightarrow \infty} R_{\mu_{q_n}}'(x+h-\epsilon)\\
       &= \lim_{h \downarrow 0,h\in\mathcal H} \lim_{n \rightarrow \infty} R_{\mu_{q_n}}'(x+h-\epsilon) \\
       &= \lim_{h \downarrow 0} R_{\mu_t}'(x+h-\epsilon)=R_{\mu_t}'(x-\epsilon),     
    \end{align*}
   where the last but one equality follows from~\eqref{distcont}.
   Similarly we see that $R_{\nu_t}'(x) \leq R_{\mu_t}'(x+\epsilon)$. 
   We have shown that $\nu_t \in B^\infty(\mu_t, \epsilon)$ for all $t \in [0,1]$. 
   From the definition of $\nu_t$ we have $\nu_q \leqc \nu_t$ for $q<t, q\in\mathcal D$ 
   and $\nu_t \leqc \nu_p$ for $p>t, p\in\mathcal D$. 
   This implies $\nu_s \leqc \nu_t$ for all $0 \leq s \leq t \leq 1$, and thus $(\nu_t)_{t \in [0,1]}$ is a peacock with mean~$m$.
\end{proof}

\section{Approximation by peacocks: stop-loss distance}  \label{sec:SL}
%\section{Strassen's theorem for the stop-loss distance}  \label{sec:SL}

The stop-loss distance~\cite{DeDh92,Ge79,KaVaGo88}
is defined as
\[
   d^{\mathrm{SL}}(\mu, \nu) = \sup_{x \in \setR} \big|R_{\mu}(x)-R_{\nu}(x) \big|,
   \quad \mu,\nu \in \setM.
\]
We will denote closed balls with respect to~$d^{\mathrm{SL}}$ by~$B^{\mathrm{SL}}$. 
In the following proposition, we use the same notation for least elements as in the case
of the infinity Wasserstein distance; no confusion should arise.
\begin{prop}\label{LSdsl}
Given $\epsilon>0$, a measure $\mu \in \setM$ and $m \in [\mathbb{E}\mu-\epsilon, \mathbb{E}\mu+\epsilon]$, there exists a unique measure $S(\mu) \in B^{\mathrm{SL}}(\mu,\epsilon)\cap \setM_m$, such that 
\[
S(\mu) \leqc \nu, \quad \mbox{for all} \ \nu \in B^{\mathrm{SL}}(\mu,\epsilon)\cap \setM_m.
\]
The call function of $S(\mu)$ is given by
\begin{equation}\label{eq:sl min}
  R^{\min}_\mu(x)=R_{S(\mu)}(x)=\bigl(m-x \bigr)^+ \vee \bigl(R_\mu(x)-\epsilon\bigr).\\
\end{equation}
To highlight the dependence on $\epsilon$ and $m$ we will sometimes write $S(\mu; m, \epsilon)$ or $R^{\min}_\mu(\: \cdot \:; m, \epsilon)$. 
\end{prop}
\begin{proof}
It is easy to check that $R_{S(\mu)}$ defines a call function, and by \ref{limitcall} of Proposition \ref{RFprop} we have
\begin{align*}
  \mathbb{E}R_{S(\mu)} &= \lim_{x\rightarrow -\infty} R_{S(\mu)}(x)+x \\
  &=\lim_{x\rightarrow -\infty} \bigl(m \vee \bigl(R_\mu(x)+x-\epsilon\bigr)\bigr) \\
  &= m \vee \bigl(\mathbb{E}\mu-\epsilon\bigr)=m.
\end{align*}
The rest is clear. 
\end{proof}
\begin{rem}
The set $B^{\mathrm{SL}}(\mu,\epsilon)\cap \setM_m$ does not contain a greatest element. To see this, take an arbitrary $\nu \in B^{\mathrm{SL}}(\mu,\epsilon)\cap \setM_m$ and define $x_0 \in \setR$ as the unique solution of $R_\nu(x)=\frac12 \epsilon$. Then for $n \in \setN$ define new call functions
$$
R_n(x)=\begin{cases}
(x-x_0)\frac{R_\nu(x_0+n)-R_\nu(x_0)}{n}+R_\nu(x_0), & x \in [x_0,x_0+n], \\[6pt]
R_\nu(x), & \mathrm{otherwise}. \end{cases}
$$
It is easy to check that $R_n$ is indeed a call function and the associated measures $\theta_n$ lie in $B^{\mathrm{SL}}(\mu,\epsilon)\cap \setM_m$. Furthermore, from the convexity of $R_\nu$ we can deduce that $R_\nu \leq R_n \leq R_{n+1}$, and hence $\nu \leqc \theta_n \leqc \theta_{n+1}$. The call functions $R_n$ converge to a function $R$ which is not a call function since $R(x)=R_\nu(x_0)=\frac \epsilon 2 $ for all $x \geq x_0$. Therefore no greatest element can exist. 
However, it is true that a measure $\nu$ is in $B^{\mathrm{SL}}(\mu,\epsilon)$ if and only if $R^{\min}_\mu(\: . \:; \mathbb{E}\nu,\epsilon) \leq R_\nu \leq R_\mu+\epsilon$.
\end{rem}

\begin{thm} \label{mainsl}
Let $(\mu_n)_{n \in \setN}$ be a sequence in $\setM$  such that 
\begin{equation*} \label{setnichtleerSL}
I:= \bigcap_{n \in \setN}[\mathbb{E}\mu_n-\epsilon, \mathbb{E}\mu_n+\epsilon],
\end{equation*}
is not empty. Denote by $(R_n)_{n \in \setN}$ the corresponding call functions. Then there exists a peacock~$(\nu_n)_{n \in \setN}$ such that
%an $\setR$-valued martingale  $M=(M_n)_{n \in \setN}$ on a probability space  $(\Omega, \mathcal{F}, \mathbb{P})$ such that 
\begin{equation}\label{eq:sl d}
  d^{\mathrm{SL}}(\mu_n, \nu_n) \leq \epsilon, \quad n\in\setN,
\end{equation}
if and only if
\begin{equation}\label{eq:2eps}
R_{k}(x) \leq R_n(x)+2\epsilon, \quad \text{for all} \ k \leq n \ \text{and} \ x \in \setR. 
\end{equation}
%In this case it is possible to choose $\mathbb{E}\nu_1=m$.  
\end{thm}
\begin{proof}
We first argue that~\eqref{eq:2eps} is equivalent to the assertion
\begin{equation} \label{minklmaxSL}
\text{There is }m\in I\ \text{such that}\
R_{k}^{\min}(x; m, \epsilon) \leq R_n(x)+\epsilon, \quad \text{for all} \ k \leq n \ \text{and} \ x \in \setR,
\end{equation}
where $R_k^{\min}$ denotes the call function of $S(\mu_k; m, \epsilon)$.
Indeed, by~\eqref{eq:sl min}, \eqref{minklmaxSL} clearly implies~\eqref{eq:2eps},
and the converse implication follows from the obvious estimate
$(m-x)^+\leq R_n(x)+\epsilon$, valid for arbitrary $m\in I$.

Now suppose that~\eqref{minklmaxSL} holds. We will define the measures $\nu_n$ via their call functions $P_n$. Define $P_1(x)=R_1^{\min}(x;m, \epsilon)$ and
\begin{equation}\label{eq:slP}
  P_n(x)= \max \bigl\{P_{n-1}(x), R_n^{\min}(x;m,\epsilon) \bigr\}, \quad n\geq 2.
\end{equation} 
It is easily verified that $P_n$ is a call function  and satisfies 
\begin{equation} \label{addsl1}
  R_n^{\min}(x) \leq P_n(x) \leq R_n(x)+\epsilon, \quad x \in \setR,
\end{equation} 
and therefore $\nu_n$, the measure associated to $P_n$, satisfies $\nu_n \in B^{\mathrm{SL}}(\mu_n, \epsilon)$. 
Furthermore $P_n \leq P_{n+1}$, and thus  $(\nu_n)_{n \in \setN}$ is a peacock with mean $m$.

Now assume that $(\nu_n)_{n \in \setN}$ is a peacock such that $d^{\mathrm{SL}}(\mu_n, \nu_n) \leq \epsilon$. We will denote the call function of $\nu_n$ by $P_n$ and set $m=\mathbb{E}\nu_1 \in I$.
Then for $k\leq n$ and $x \in \setR$ we get with Proposition~\ref{LSdsl}
$$
R_k^{\min}(x;m, \epsilon) \leq P_k(x) \leq P_n(x) \leq R_n(x)+\epsilon. 
$$
\end{proof}
Note that~\eqref{minklmaxSL} trivially holds for $k=n$. Moreover,
unwinding the recursive definition~\eqref{eq:slP} and using~\eqref{eq:sl min}, we see
that~$P_n$ has the explicit expression
\[
  P_n(x) = \max\{(m-x)^+,R_1(x)-\epsilon,\dots, R_n(x)-\epsilon\}, \quad x\in\setR, n\in \setN.
\]
The following proposition shows that the peacock from Theorem~\ref{mainsl} is never unique.
\begin{prop}
  In the setting of Theorem~\ref{mainsl}, suppose that~\eqref{minklmaxSL} holds. Then
  there are infinitely many peacocks satisfying~\eqref{eq:sl d}.
\end{prop}
\begin{proof}
  Define~$P_n$ as in the proof of Theorem~\ref{mainsl}, and fix $x_0\in\setR$ with
  $P_1(x_0)<\epsilon$. For arbitrary $c\in(0,1)$, we define
  \[
    G(x) =
    \begin{cases}
      P_1(x_0), & x\leq x_0, \\
      P_1(x_0) + c P_1'(x_0)(x-x_0), & x\geq x_0.
    \end{cases}
  \]
  Thus, in a right neighborhood of~$x_0$, the graph of~$G$ is a line that lies above~$P_1$.
  We then put $\tilde{P}_n = P_n \vee G$, for $n\in\setN$. It is easy to
  see that $(\tilde{P}_n)$ is an increasing sequence of call functions with mean~$m$,
  and thus defines a peacock. Moreover, we have
  \[
    \tilde{P}_n \leq (R_n+\epsilon) \vee G \leq R_n+\epsilon,
  \]
  by~\eqref{addsl1} and the fact that $G\leq \epsilon$.
  The lower estimate $\tilde{P}_n \geq P_n\geq R_n-\epsilon$ is also obvious.
\end{proof}
Theorem~\ref{mainsl} easily extends to continuous index sets. 
\begin{thm}\label{thm:sl01}
Assume that $(\mu_t)_{t \in [0,1]}$ is a family of measures in $\setM$ such that 
$$I:=\bigcap_{t \in [0,1]}[\mathbb{E}\mu_t-\epsilon, \mathbb{E}\mu_t+\epsilon]$$ 
is not empty. Denote the call function of $\mu_t$ by $R_t$.
Then there exists a peacock $(\nu_t)_{t \in [0,1]}$ with 
\[
  d^{\mathrm{SL}}\bigl(\mu_t, \nu_t\bigr) \leq \epsilon,  \quad \mbox{for all} \ t \in [0,1],
\]
if and only if
\begin{equation}\label{eq:2eps cont}
  R_{s}(x) \leq R_t(x)+2\epsilon, \quad \text{for all} \ 0 \leq s<t \leq 1 \ \text{and} \ x \in \setR. 
\end{equation}
%Here $R_s^{\min}$ denotes the call function of $S(\mu_s; m, \epsilon)$. In this case it is possible to choose $\mathbb{E}\nu_1=m$.  
\end{thm}
\begin{proof}
As in the discrete case (Theorem~\ref{mainsl}), \eqref{eq:2eps cont}
is equivalent to the statement
\begin{equation} \label{minklmaxSLcont}
\text{There is}\ m\in I\ \text{such that}\ R_{s}^{\min}(x; m, \epsilon) \leq R_t(x)+\epsilon, \quad \text{for all}  \ 0 \leq s<t \leq 1\ \text{and}\ x \in \setR. 
\end{equation}
If \eqref{minklmaxSLcont} holds, then we set 
\[
  P_t(x)= \sup_{s \leq t} R_s^{\min}(x;m,\epsilon) , \quad t \in [0,1].
\]
Then $P_t$ is a call function which satisfies $R_t^{\min}(x;m,\epsilon) \leq P_t(x) \leq  R_t(x)+\epsilon$
for $x\in\setR$. The rest can be done as in the proof of Theorem~\ref{mainsl}.
\end{proof}

\section{L\'evy distance and Prokhorov distance: preliminaries} \label{sec:LP}

%We will  begin with the definition the L\'evy distance and the Prokhorov distance.
The L\'evy distance is a metric on the set of all measures on $\setR$, defined as
\[
   d^\mathrm{L}(\mu,\nu)=\inf \Bigl\{h >0: F_\mu(x-h)-h \leq F_\nu(x) \leq F_\mu(x+h)+h,\ \forall x \in \setR \Bigr\}.
\]
Its importance is partially due to the fact that~$d^\mathrm{L}$ metrizes weak
convergence of measures on~$\setR$. 
The Prokhorov distance is a metric on measures on an arbitrary separable metric space $(S, \rho)$. For measures $\mu, \nu$ on $S$ it can be written as
\[
   d^\mathrm{P}(\mu,\nu)=\inf \Bigl\{h >0: \nu(A) \leq \mu(A^h)+h, \ \text{for all closed sets} \ A \subseteq S \Bigr\},
\]
where $A^h= \bigl\{x \in S: \inf_{a \in A} \rho(x,a) \leq h \bigr\}$. 
The Prokhorov distance is often referred to as a generalization of the L\'evy metric, since $d^\mathrm{P}$ metrizes weak convergence on any separable metric space. 
Note, though, that $d^\mathrm{L}$ and $d^\mathrm{P}$ do not coincide when $(S,\rho)=(\setR, |\:.\:|)$.
%as  shown in the following example.
%
%
%\begin{exa}
%   Let $\epsilon=\frac 18$, $\mu$ be the uniform distribution on $[0,1]$, and $\nu$ be the uniform distribution on $[2\epsilon,1-2\epsilon]$. 
%   Then it is easy to check that $d^\mathrm{L}(\mu,\nu) \leq \epsilon$. Also we have 
%   \[
%      F_\mu\Bigl(\frac 14 -\epsilon\Bigr)-\epsilon=F_\nu\Bigl(\frac 14\Bigr),
%   \]
%   hence $d^\mathrm{L}(\mu,\nu) = \epsilon$. 
%   Next, we will show that the Prokhorov distance of $\mu$ and $\nu$ is larger than~$\frac 16$, and hence not equal to the L\'evy distance.
%   Consider the closed set $B=[2\epsilon,1-2\epsilon]$. Then $\nu(B)=1$, and the inequality
%   \begin{align*}
%     1 &\leq \mu(B^h)+h \\
%     &=\mu\bigl([2\epsilon-h,1-2\epsilon+h]\bigr)+h= 1-4 \epsilon +3h
%   \end{align*}
%   is true for all $h \geq \frac 16$, and therefore $d^\mathrm{P}(\mu,\nu) \geq \frac 16$.
%\end{exa}
It is easy to see (\cite{HuRo09}, p.~36) that the Prokhorov distance of two measures on $\setR$ is an upper bound for their L\'evy distance:
\begin{lem} \label{lem:levyprokohorov}
   Let $\mu$ and $\nu$ be two probability measures on $\setR$. Then $d^\mathrm{L}(\mu,\nu) \leq d^\mathrm{P}(\mu,\nu)$.
\end{lem}
For further information concerning these metrics, their properties and their relations to other metrics,
we refer the reader to \cite{HuRo09}~(p.27~ff).
Now we define slightly different distances $d_p^\mathrm{L}$ and $d_p^\mathrm{P}$ on the set of probability measures on~$\setR$, which in general are not metrics in the classical sense
(recall the remark after Definition~\ref{infwassersteindistance}).
These distances are useful for two reasons: 
 First, it will turn out that balls with respect to $d^\mathrm{L}$ and $d^\mathrm{P}$ can always
 be written as balls w.r.t.\ $d_p^\mathrm{L}$ and $d_p^\mathrm{P}$, see Lemma \ref{lem:eqballs}. Second, the function $d_p^\mathrm{P}$ has a direct link to minimal distance couplings which are especially useful for applications, see Proposition \ref{prop:DudleyStrassen}.
For $p \in [0,1]$ we define 
\begin{equation} \label{pmetrik}
   d_p^\mathrm{L}(\mu,\nu):= \inf\Bigl\{h > 0: \ F_\mu(x-h) -p \leq F_\nu(x) \leq F_\mu(x+h)+p, \forall x \in \setR \Bigr\}
\end{equation}
and 
\begin{equation} \label{pmetrik2}
   d_p^\mathrm{P}(\mu,\nu):= \inf \Bigl\{h >0: \nu(A) \leq \mu(A^h)+p, \ \text{for all closed sets} \ A \subseteq S \Bigr\}.
\end{equation}
It is easy to show (using complements) that $d_p^\mathrm{P}(\mu,\nu)=d_p^\mathrm{P}(\nu,\mu)$ (see e.g.~Proposition~1 in~\cite{Du68}). 
%Furthermore~$d_p^\mathrm{L}$ and~$d_p^\mathrm{P}$ are not metrics if $p>0$, since 
Note that $d_p^\mathrm{P}(\mu, \nu)=0$ does not imply that $\mu=\nu$. 
%Nevertheless we will call sets of the form $B_p^\mathrm{L}(\mu,\epsilon)=\{\theta \in \setM: d_p^\mathrm{L}(\mu,\theta)\leq \epsilon\}$ balls with respect to $d_p^\mathrm{L}$ 
%and similarly we define closed ``balls'' with respect to~$d_p^\mathrm{P}$. 
We will refer to~$d_p^\mathrm{L}$ as the modified L\'evy distance, and to~$d_p^\mathrm{P}$ as the modified Prokhorov distance.\footnote{Note that our definition of the modified Prokhorov distance does \emph{not} agree with
the Prokhorov-type metric $\pi_\lambda$ from~\cite{Ra91} and~\cite{RaRuSc92}.}
The following Lemma explains the connection between the L\'evy distance $d^\mathrm{L}$ and the modified L\'evy distance  $d_p^\mathrm{L}$,
resp.\ the Prokhorov distance $d^\mathrm{P}$ and the modified  Prokhorov distance $d_p^\mathrm{P}$.
% We note it separately because of its importance, its proof is trivial.

\begin{lem} \label{lem:eqballs}
   Let $\mu \in \setM$. Then for every $\epsilon \in [0,1]$ we have 
   \[
     B^\mathrm{L}(\mu,\epsilon)=B_\epsilon^\mathrm{L}(\mu,\epsilon) \quad \text{and} \quad  B^\mathrm{P}(\mu,\epsilon)=B_\epsilon^\mathrm{P}(\mu,\epsilon).
   \]
\end{lem}
\begin{proof}
   For $\nu\in\setM$, the assertion $\nu\in B^\mathrm{P}(\mu,\epsilon)$ is equivalent to
  \begin{equation}\label{eq:P}
     \mu(A) \leq \nu\bigl(A^{\epsilon+\delta}\bigr)+\epsilon+\delta, \quad \delta>0,\ A\subseteq \setR \ \text{closed},
  \end{equation}
  whereas $\nu\in B_{\epsilon}^\mathrm{P}(\mu,\epsilon)$ means that
  \begin{equation} \label{eq:P eps}
     \mu(A) \leq \nu\bigl(A^{\epsilon+\delta}\bigr)+\epsilon, \quad \delta>0,\ A\subseteq \setR \ \text{closed}.
  \end{equation}
   Obviously, \eqref{eq:P eps} implies~\eqref{eq:P}. Now suppose that~\eqref{eq:P} holds,
  and let $\delta\downarrow0$. Notice that $A^{\epsilon+\delta_1} \subseteq A^{\epsilon+\delta_2}$ for $\delta_1 \leq \delta_2$. The continuity of $\nu$ then gives
  \begin{equation*} 
     \mu(A) \leq \nu\bigl(A^{\epsilon}\bigr)+\epsilon \leq \nu\bigl(A^{\epsilon+\delta}\bigr)+\epsilon \quad \delta>0,\ A\subseteq \setR \ \text{closed},
  \end{equation*}
  and thus $B^\mathrm{P}(\mu,\epsilon)=B_\epsilon^\mathrm{P}(\mu,\epsilon)$. 
  Replacing $A$ by intervals $(-\infty,x]$ for $x \in \setR$ in~\eqref{eq:P} and~\eqref{eq:P eps}
	proves that $B^\mathrm{L}(\mu,\epsilon)=B_\epsilon^\mathrm{L}(\mu,\epsilon)$.
  %For $\nu\in\setM$, the assertion $\nu\in B^\mathrm{L}(\mu,\epsilon)$ is equivalent to
  %\begin{equation}\label{eq:L}
     %F_{\mu}(x-(\epsilon+\delta))-(\epsilon+\delta) \leq F_{\nu}(x) \leq
      %F_\mu(x+\epsilon+\delta)+\epsilon+\delta, \quad \delta>0, x \in \setR,
  %\end{equation}
  %whereas $\nu\in B_{\epsilon}^\mathrm{L}(\mu,\epsilon)$ means that
  %\begin{equation} \label{eq:L eps}
    %F_{\mu}(x-(\epsilon+\delta))-\epsilon \leq F_{\nu}(x) \leq
      %F_\mu(x+\epsilon+\delta)+\epsilon, \quad \delta>0, x \in \setR.
  %\end{equation}
  %Obviously, \eqref{eq:L eps} implies~\eqref{eq:L}. Now suppose that~\eqref{eq:L} holds,
  %and let $\delta\downarrow0$ in its second inequality. Then, since $F_\mu$ increases,
  %we obtain the second inequality of~\eqref{eq:L eps}. 

  %If the first inequality
  %of~\eqref{eq:L eps} does \emph{not} hold, then there are $x_0,\delta_0$
  %such that
  %\[
    %\delta_1 :=\tfrac12 \min\{\delta_0, 
      %F_{\mu}(x_0-(\epsilon+ \delta_0))-\epsilon-F_{\nu}(x_0)\} >0.
  %\]
  %Then the estimate
  %\begin{align*}
     %F_{\mu}(x_0-(\epsilon+ \delta_1))-(\epsilon+ \delta_1)
       %&\geq F_{\mu}(x_0-(\epsilon+ \delta_0))-(\epsilon+  \delta_1) \\
       %&> F_{\mu}(x_0-(\epsilon+ \delta_0))-\epsilon
         %- \big(F_{\mu}(x_0-(\epsilon+ \delta_0))-\epsilon-F_{\nu}(x_0)\big) \\
      %&= F_{\nu}(x_0)
  %\end{align*}
  %shows that the first inequality of~\eqref{eq:L} does not hold, either.
\end{proof}

Similarly to Lemma \ref{lem:levyprokohorov} we can show that the modified L\'evy distance of two measures never exceeds the modified Prokhorov distance.
\begin{lem} \label{lem:levyprokohorov_p}
   Let $\mu$ and $\nu$ be two probability measures on $\setR$ and let $p \in [0,1]$. Then
   \[
      d_p^\mathrm{L}(\mu,\nu) \leq d_p^\mathrm{P}(\mu,\nu).
   \]
\end{lem}
\begin{proof}
   We set $\epsilon=d_p^\mathrm{P}(\mu,\nu)$. Then for any $x \in \setR$ and all $n \in \setN$ we have 
   \begin{align*}
     F_\nu(x) = \nu\bigl((-\infty,x]\bigr) 
     &\leq \mu\Big(\Big(-\infty,x+\epsilon+\frac 1n\Big]\Big)+p \\ 
     &= F_\mu\Big(x+\epsilon+ \frac 1n \Big)+p,
   \end{align*}
   and by the symmetry of $d^\mathrm{P}$ the above relation also holds with $\mu$ and $\nu$ interchanged. This implies that $d_p^\mathrm{L}(\mu,\nu) \leq \epsilon$.
\end{proof}
The following coupling representation of $d_p^\mathrm{P}$
was first proved by Strassen and was then extended by Dudley~\cite{Du68,St65}.
\begin{prop} \label{prop:DudleyStrassen}
   Given measures $\mu, \nu$ on $\setR$, $p \in [0,1]$, and $\epsilon>0$ there exists a probability space~$(\Omega, \mathcal{F}, \pp)$ with random variables $X\sim \mu$ and $Y \sim \nu$ such that 
   \begin{equation} \label{smalldist} 
     \pp\big( \bigl|X-Y|> \epsilon  \big) \leq p,
   \end{equation}
   if and only if
    \begin{equation} 
     d^\mathrm{P}_p(\mu,\nu) \leq \epsilon.
   \end{equation}
\end{prop}

\section{Approximation by peacocks: Prokhorov distance and L\'evy distance}\label{se:thm P L}
%\section{Strassen's theorem for Prokhorov distance and L\'evy distance}\label{se:thm P L}

In this section we will prove peacock approximation results, first for the modified Prokhorov distance 
and later on for the modified L\'evy distance, the Prokhorov distance, and the L\'evy distance. 
It turns out that Problem~\ref{pr:main} always has a solution for these distances,
regardless of the size of~$\epsilon$.
In the following we denote the quantile function of a measure $\mu \in \setM$ by $G_\mu$, i.e.
\[
   G_\mu(p)=\inf\left\{ x\in \setR: \: F_\mu(x) \geq p \right\}, \quad p \in [0,1].
\]

\begin{prop} \label{prop:dppempty}
   Let $\mu \in \setM$, $p \in (0,1]$, and $m \in \setR$. Then the set 
   \[
   B_p^{\mathrm{P}}(\mu,0)\cap\setM_m 
   \]
   is not empty. Moreover, this set contains at least one measure with bounded support.
\end{prop}

\begin{proof}
   The statement is clear for $p=1$, and so so we focus on $p \in (0,1)$.    
   Given a measure $\mu$ we set $I=[G_\mu\bigl(\frac p4 \bigr), G_\mu\bigl(1- \frac p4 \bigr))$. We will first define a measure~$\eta$ 
   with bounded support which lies in $B_p^{\mathrm{P}}(\mu,0)$, and then we will modify it to obtain a measure~$\theta$ with mean~$m$. 
   We set
   \[
     F_\eta(x):= \begin{cases} 
        0, & x < G_\mu\bigl(\frac p4 \bigr),  \\
        F_\mu(x), & x \in I,\\
        1, & x \geq G_\mu\bigl(1- \frac p4 \bigr),
     \end{cases}
   \]
   which is clearly a distribution function of a measure $\eta$. Note that $\eta$ has bounded support, 
   so in particular~$\eta$ has finite mean. Next we define
   \[
     \theta= \Bigl(1-\frac p2 \Bigr) \eta + \frac p2 \delta_w,
   \]
   where $w$ is chosen such that $\mathbb{E}\theta=m$. Since~$\eta$ has bounded support, we can deduce that~$\theta$ also has bounded support.
   Now for every closed set $A \subseteq \setR$ we have
   \begin{align*}
    \theta(A) &\leq \bigl(1-\frac p2 \bigr) \eta(A) + \frac p2 \\
     &\leq \bigl(1-\frac p2 \bigr) \eta\bigl(A \cap \mathrm{int}(I)\bigr) + p  \\
     &\leq \mu(A) + p,
   \end{align*}
   where $\mathrm{int}(I)$ denotes the interior of $I$. For the last inequality, note that $\mu$ and $\eta$ are equal on $\mathrm{int}(I)$.
   The last equation implies that $\theta \in B_p^{\mathrm{P}}(\mu,0)\cap\setM_m$. 
\end{proof}

   Note that in Proposition \ref{prop:dppempty} it is not important that $\mu$ has finite mean. The statement is true for all measures on $\setR$.
   The same is true for all subsequent results of this section.

\begin{prop} \label{prop:dpplarge}
   Let $\nu \in \setM$ be a measure with bounded support and $p \in (0,1)$. Then for all measures $\mu \in \setM$ 
   there exists a measure $\theta \in B_p^\mathrm{P}(\mu,0)$ with bounded support such that~$\nu \leqc \theta$.
\end{prop}

\begin{proof}
   Fix $\mu, \nu \in \setM$ and $p \in (0,1)$, and set $m=\mathbb{E}\nu$. Then, by
   Proposition~\ref{prop:dppempty}, there is a measure 
   $\theta_0 \in B_{p/2}^{\mathrm{P}}(\mu,0)\cap\setM_m$ which has bounded support. For $n \in \setN$ we define
   \[
     \theta_n= \bigl(1-\frac p2\bigr) \theta_0 + \frac p4 \delta_{m-n} + \frac p4 \delta_{m+n}.
   \]
   These measures have bounded support and mean $m$. Furthermore, for $A\subseteq \setR$ closed, we have
   \begin{align*}
     \theta_n(A) &\leq \bigl(1-\frac p2\bigr) \theta_0(A) + \frac p2 \\
      &\leq \theta_0(A)+ \frac p2 \leq \mu(A)+p, \quad n\in\setN,
   \end{align*}
   and hence $\nu_n \in B_p^\mathrm{P}(\mu,0)$ for all $n \in \setN$.
   Now observe that for all $n\in\setN$ and $x \in [m-n,m+n]$ we have 
   \begin{equation}\label{eq:Rninfty}
   R_{\theta_n}(x)= \bigl(1-\frac p2\bigr)R_{\theta_0}(x) + \frac p4 \bigl(m+n-x),
   \end{equation}
   which tends to infinity as~$n$ tends to infinity. 
   %NEU
   Outside of the support of~$\theta_n$ (i.e.\ outside the interval $[m-n,m+n]$) 
   the call function of~$\theta_n$ equals the call function of the Dirac measure~$\delta_m$ with mass at~$m$.
   Therefore there has to exist $n_0 \in \setN$ such that $\nu \leqc \theta_{n_0}$.
\end{proof}

   In Proposition \ref{prop:dpplarge} it is important that $p>0$. For $p=0$ the limit in~\eqref{eq:Rninfty} is finite.

\begin{thm} \label{thm: dppstrassen}
   Let $(\mu_n)_{n \in \setN}$ be a sequence in $\setM$, $\epsilon>0$, and $p \in (0,1]$. Then, for all $m \in \setR$ there exists a peacock $(\nu_n)_{n \in \setN}$ with mean $m$ such that
   \[
     d_p^\mathrm{P}(\mu_n, \nu_n) \leq \epsilon.
   \]
\end{thm}

\begin{proof}
   If $p=1$ then $B_p^\mathrm{P}(\mu,0)$ contains all probability measures on $\setR$, which is easily seen from the definition of $d_p^\mathrm{P}$, and the result is trivial.   
   So we consider the case $p<1$. Since $B_p^\mathrm{P}(\mu,0) \subseteq B_p^\mathrm{P}(\mu,\epsilon)$, it suffices to prove the statement for $\epsilon=0$.
   By Proposition~\ref{prop:dppempty}, there exists a measure $\nu_1 \in B_p^\mathrm{P}(\mu_1,0)\cap\setM_m$ with bounded support.
   By Proposition~\ref{prop:dpplarge} there exists a measure $\nu_2 \in B_p^\mathrm{P}(\mu_2,0)$ 
   such that $\nu_1 \leqc \nu_2$. Since $\nu_2$ has again finite support, we can proceed inductively to finish the proof.
\end{proof}

Setting~$\epsilon=p \in (0,1]$ in the previous result, we obtain the following corollary.
\begin{cor} \label{cor:strassendp}
   Let $(\mu_n)_{n \in \setN}$ be a sequence in $\setM$ and $\epsilon > 0$. Then, for all $m \in \setR$ there exists a peacock $(\nu_n)_{n \in \setN}$ with mean $m$ such that
   \[
     d^\mathrm{P}(\mu_n, \nu_n) \leq \epsilon.
   \]
\end{cor}

\begin{proof}
   By Lemma \ref{lem:eqballs} we have $B^\mathrm{P}(\mu,\epsilon)=B_\epsilon^\mathrm{P}(\mu,\epsilon)$ for all $\mu \in \setM$ and $\epsilon \in [0,1]$.
   The result now easily follows from Theorem~\ref{thm: dppstrassen}.
\end{proof}

Since balls with respect to the modified Prokhorov metric are smaller than balls with respect to the L\'evy metric, we get the following corollary. 

\begin{thm} \label{thm: dpLstrassen}
   Let $(\mu_n)_{n \in \setN}$ be a sequence in $\setM$, $\epsilon>0$, and $p \in (0,1]$. Then, for all $m \in \setR$ there exists a peacock $(\nu_n)_{n \in \setN}$ with mean $m$ such that
   \[
     d_p^\mathrm{L}(\mu_n, \nu_n) \leq \epsilon.
   \]
   In particular, there exists a peacock $(\nu_n)_{n \in \setN}$ with mean $m$ such that
   \[
     d^\mathrm{L}(\mu_n, \nu_n) \leq \epsilon.
   \]
\end{thm}
\begin{proof}
   Fix $\epsilon>0$ and $p \in (0,1]$, and let $(\nu_n)_{n \in \setN}$ be the peacock from Theorem~\ref{thm: dppstrassen} resp.\ Corollary~\ref{cor:strassendp}. 
   Then by Lemma \ref{lem:levyprokohorov_p} resp.\ Lemma \ref{lem:levyprokohorov}, we have 
   $\nu_n \in B_p^\mathrm{L}(\mu_n,\epsilon)$ resp. $\nu_n \in B^\mathrm{L}(\mu_n,\epsilon)$ for all $n \in \setN$. 
\end{proof}

   For $\mu \in \setM$, $\epsilon \geq 0$, $p \in (0,1)$, and $m \in \setR$, the set $B_p^\mathrm{L}(\mu,\epsilon)\cap \setM_m$ always contains a least element with respect to $\leqc$, with an explicit
   call function. See Section~2.4.3 in~\cite{Gu16}.

\section{A variant of Strassen's theorem}\label{se:new Strassen}

So far, we discussed the problem of approximating a given sequence of measures~$(\mu_n)$ by a
peacock~$(\nu_n)$.
If the distance is measured by $W^\infty$, then the existence of such a peacock has two
consequences: First, there is a probability space with a martingale~$M^*$
with marginals~$(\nu_n)$ (by Strassen's theorem). Second, the definition of~$W^\infty$
implies that for each $\epsilon'>\epsilon$
there is a probability space supporting processes $\hat M$ and $\hat X$
satisfying $\mathbb{\hat P}[|\hat{M}_n-\hat{X}_n|>\epsilon']=0$ for all~$n$.
It is now a natural question whether a martingale~$M$ with marginals~$(\nu_n)$ can be found
such that there is an adapted process~$X$ satisfying $\mathbb{P}[|M_n-X_n|>\epsilon']=0$.
We answer this question affirmatively for finite sequences of measures with finite support.
This restriction suffices for the financial application that motivated our study (see~\cite{GeGu16}),
and it allows to replace ``for  all $\epsilon'>\epsilon$ $\dots$ $\mathbb{P}[|M_n-X_n|>\epsilon']=0$''
simply by $\mathbb{P}[|M_n-X_n|>\epsilon]=0$.
The result (Theorem~\ref{thm:strassenW}) is a consequence of Theorem~\ref{strassenwinf}
and the following lemma.

\begin{lem}\label{le:finite}
  Let $\epsilon>0$.
  Let $(\nu_n)_{n=1,\dots,n_0}$ be a peacock, and  $(\mu_n)_{n=1,\dots,n_0}$ be a sequence
  of measures in~$\mathcal M$. %, all (both $\nu_n$ and $\mu_n$) with finite support.
  %TODO Endlichkeit von Omega^* ergibt sich nicht aus Strassen, deshalb Omega^* in Voraussetzung
  Assume that there is a finite filtered probability space
  $(\Omega^*,\mathcal{F}^*,(\mathcal{F}^*_n)_{1\leq n\leq n_0},
    \mathbb{P}^*)$ with a martingale~$M^*$ satisfying
  $M_n^*\sim \nu_n$ for $1\leq n\leq n_0$.
  
  Assume further that there is a finite probability space $(\hat \Omega,\hat{\mathcal F},\hat{\mathbb P})$
  supporting two processes~$\hat M$ and~$\hat X$ satisfying $\hat{M}_n\sim \nu_n$,
  $\hat{X}_n\sim \mu_n$ for $1\leq n\leq n_0$ and
  \begin{equation}\label{eq:XM}
    \hat{\mathbb P}[|\hat{M}_n-\hat{X}_n|>\epsilon]=0, \quad n=1,\dots,n_0.
  \end{equation}
  Then there is a finite filtered probability space $(\Omega,\mathcal{F},
  (\mathcal{F}_n)_{1\leq n\leq n_0},\mathbb P)$
  with processes~$M$ and~$X$ combining all properties mentioned, i.e.:
  \begin{itemize}
    \item $M$ is a martingale
    \item $X$ is adapted
    \item $M_n\sim \nu_n$,  $X_n\sim \mu_n$, $\quad n=1,\dots,n_0$,
    \item $\mathbb{P}[|M_n-X_n|>\epsilon]=0, \quad n=1,\dots,n_0$.
  \end{itemize}
\end{lem}
\begin{proof}
  %W.l.o.g.\ we assume that $(\mathcal F_n)$ is the filtration generated by~$M$.
  Let $n'\in\{1,\dots,n_0\}$ and assume, inductively,
  that we have already constructed a filtered probability space
  $(\Omega,\mathcal{F},(\mathcal{F}_n),\mathbb P)$
  that satisfies the requirements, where the conditions concerning~$X$ hold
  for $n<n'$, i.e.\ there are processes~$M=(M_n)_{1\leq n \leq n_0}$
  and $X=(X_n)_{1\leq n<n'}$
  such that
  \begin{itemize}
    \item $M$ is a martingale
    \item $X$ is adapted
    \item $M_n\sim \nu_n, \quad n=1,\dots,n_0$,
    \item $X_n\sim \mu_n, \quad 1\leq n<n'$,
    \item $\mathbb{P}[|M_n-X_n|>\epsilon]=0, \quad 1\leq n<n'$.
  \end{itemize}
  Note that in case $n'=1$ (induction base) we may simply take
  $(\Omega,\mathcal{F},(\mathcal{F}_n),\mathbb P)=(\Omega^*,\mathcal{F}^*,(\mathcal{F}^*_n),\mathbb{P}^*)$.
  Let $z\in \mathbb R$ be an arbitrary member of the image of~$M_{n'}$, and define
  \[
    U:= A_1 \cup \dots \cup A_m := (M_{n'})^{-1}(z),
  \]
  where $A_1,\dots,A_m$ are (distinct) atoms of~$\mathcal{F}_{n'}$. We denote the
  preimage of~$z$ in~$\hat \Omega$ by
  \[
    \{\hat{\omega}_1,\dots,\hat{\omega}_l\} := \hat{M}_{n'}^{-1}(z).
  \]
  As $M_{n'}\sim \mu_{n'}\sim \hat{M}_{n'}$, we have
  \begin{equation}\label{eq:PU}
    \mathbb{P}[U]=\mathbb{P}[A_1 \cup \dots \cup A_m] 
      = \hat{\mathbb P}[\{\hat{\omega}_1,\dots,\hat{\omega}_l\}].
  \end{equation}
  To make room for an appropriate~$X_{n'}$ on a new filtered probability space, whose
  constituents will be denoted by $\Omega'$, $\mathcal{F}'$ etc.,
  we divide each ``old'' atom
  \[
     A_r =: \{ \omega_{r1},\dots,\omega_{r k_r} \}, \quad 1\leq r\leq m,
  \]
  into~$l$ ``new'' atoms
  \[
    A_r^{(i)}:=\{\omega^{(i)}_{r1},\dots,\omega^{(i)}_{rk_r}\},\quad 1\leq r\leq m,\ 1\leq i\leq l.
  \]
  Then, define
  \[
    \Omega' := (\Omega\setminus U)\cup \bigcup_{\substack{1\leq r\leq m\\ 1\leq i\leq l}} A_r^{(i)}
  \]
  and $\mathcal{F}':=2^{\Omega'}$.
  We let $\mathbb{P}':=\mathbb{P}$ on $\Omega\setminus U$ and
  \[
    \mathbb{P}'[\omega^{(i)}_{rj}] := \frac{\mathbb{P}[\omega_{rj}]\, \hat{\mathbb{P}}[\hat{\omega}_i]}
    {\sum_{i'=1}^l \hat{\mathbb P}[\hat{\omega}_{i'}]},\quad 1\leq r\leq m,\ 1\leq i\leq l,\ 1\leq j\leq k_r.
  \]
  The sigma-algebra $\mathcal{F}'_{n'}$ is generated by the atoms of~$\mathcal{F}_{n'}$,
  but with each atom~$A_r$ replaced by the atoms $A_r^{(1)},\dots,A_r^{(l)}$. Similarly,
  we define $\mathcal{F}'_{n}$ for $n<n'$ and $n>n'$.
  E.g., if~$A_1$ decomposes into atoms $A_1=B\cup \tilde{B}$
  in $\mathcal{F}_{n'+1}$, then we replace~$B$ and $\tilde{B}$ by
  $B\cap A_1^{(1)},\dots,B\cap A_1^{(l)}$ and $\tilde{B}\cap A_1^{(1)},\dots,\tilde{B}\cap A_1^{(l)}$,
  respectively, and so on. Clearly, this defines a filtered probability space
  $(\Omega',\mathcal{F}',(\mathcal{F}'_n),\mathbb{P}')$. On this space, we define~$M'$
  like~$M$, forgetting that the atoms~$A_1,\dots,A_m$ were split: $M_n':=M_n$ for all~$n$
  on $\Omega\setminus U$ and
  \[
    M_n'(\omega_{rj}^{(i)}) := M_n(\omega_{rj}),\quad 
    1\leq r\leq m,\ 1\leq i\leq l,\ 1\leq j\leq k_r,\ 1\leq n\leq n_0.
  \]
  Thus, the adapted process
  $M'$ has the same marginal laws as~$M$. Now we verify that $M'$ is a martingale.
  Let $n_1>n'$. (The cases of time points $n_1,n_2$ in other positions relative to~$n'$
  work very similarly,
  but need additional cumbersome notation.) First, let~$A'$ be any atom of $\mathcal{F}'_{n'}$ distinct
  from $A_r^{(1)},\dots,A_r^{(l)}$, $1\leq r\leq m$. Then we compute
  \begin{align*}
    \mathbb{E}'[M'_{n_1} | A'] &=
      %\frac{\mathbb{E}'[M'_{n_1} \mathbf{1}_{A'}]}{\mathbb{P}'[A']} =
       \frac{\sum_{\omega\in A'}M'_{n_1}(\omega)\mathbb{P}'[\omega]}
        {\mathbb{P}'[A']} \\
    &= \frac{\sum_{\omega\in A'}M_{n_1}(\omega)\mathbb{P}[\omega]}
        {\mathbb{P}[A']}\\
    &= \mathbb{E}[M_{n_1} |A'] = M_{n'}(A')= M'_{n'}(A').
  \end{align*}
  For $r\in\{1,\dots,m\}$ and $i\in\{1,\dots,l\}$, we have
  \begin{align*}
    \mathbb{E}'[M'_{n_1} | A_r^{(i)}] &=
      \frac{\sum_{j=1}^{k_r} M'_{n_1}(\omega_{rj}^{(i)})\mathbb{P}'[\omega_{rj}^{(i)}]}
        {\sum_{j=1}^{k_r}\mathbb{P}'[\omega_{rj}^{(i)}]} \\
      &=
      \frac{\sum_{j=1}^{k_r} M_{n_1}(\omega_{rj})\mathbb{P}[\omega_{rj}]}
        {\sum_{j=1}^{k_r}\mathbb{P}[\omega_{rj}]}\\
      &= \mathbb{E}[M_{n_1}|A_r] = M_{n'}(A_r) = M'_{n'}(A_r^{(i)}).
  \end{align*}
  Therefore, $M'$ is a martingale. Now we define the process $(X'_n)_{1\leq n< n'}$
  as $X$ on $\Omega\setminus U$, and
  \[
    X_{n}'(\omega_{rj}^{(i)}) := X_n(\omega_{rj}),\quad 1\leq r\leq m,\ 1\leq i\leq l,\ 1\leq j\leq k_r,\
      1\leq n<n'.
  \]
  As for~$n'$, we put
  \begin{equation}\label{eq:def X'}
    X_{n'}'(\omega_{rj}^{(i)}):=\hat{X}_{n'}(\hat{\omega}_i),\quad 
    1\leq r\leq m,\ 1\leq i\leq l,\ 1\leq j\leq k_r.
  \end{equation}
  To make the definition complete, let $X_{n'}':=M_{n'}$ on $\Omega\setminus U$, although
  this is of no relevance, because this definition will be overwritten when we continue
  the construction for the next element of the image of~$M_{n'}$.
  As the right hand side of~\eqref{eq:def X'} is independent of~$j$, the process $(X'_n)_{1\leq n\leq n'}$
  is adapted to~$(\mathcal{F}'_n)_{1\leq n\leq n'}$.
  We now show that the random variables
  \[
    X'_{n'}|_U \quad \text{and} \quad
    \hat{X}_{n'}|_{\{ \hat{\omega}_1,\dots,\hat{\omega}_l \}}
  \]
  have the same law. Indeed, for $1\leq i\leq l$ we have
  \begin{align*}
     \mathbb{P}'\big[X'_{n'}|_U &= \hat{X}_{n'}(\hat{\omega}_i)\big]
     = \sum_{r=1}^m \sum_{j=1}^{k_r}
       \sum_{i':\, X'_{n'}(\omega_{rj}^{(i')})=\hat{X}_{n'}(\hat{\omega}_i) }
       \mathbb{P}'[\omega_{rj}^{(i')}] \\
     &= \sum_{r=1}^m \sum_{j=1}^{k_r}
     \sum_{i':\, \hat{X}_{n'}(\hat{\omega}_{i'})=\hat{X}_{n'}(\hat{\omega}_i) }
        \frac{\hat{\mathbb P}[\hat{\omega}_{i'}]\, \mathbb{P}[\omega_{rj}]}
         {\sum_{i''=1}^l \hat{\mathbb P}[\hat{\omega}_{i''}]} \\
     &= \frac{1}{\sum_{i''=1}^l \hat{\mathbb P}[\hat{\omega}_{i''}]}
     \sum_{r=1}^m \sum_{j=1}^{k_r} \mathbb{P}[\omega_{rj}]
     \sum_{i':\, \hat{X}_{n'}(\hat{\omega}_{i'})=\hat{X}_{n'}(\hat{\omega}_i) }
        \hat{\mathbb P}[\hat{\omega}_{i'}] \\
     &= \hat{\mathbb P}\big[
     \hat{X}_{n'}|_{\{ \hat{\omega}_1,\dots,\hat{\omega}_l \}} = \hat{X}_{n'}(\hat{\omega}_i) \big],
  \end{align*}
  where we used~\eqref{eq:PU} in the last inequality.
  It remains to verify
  \[
    \mathbb{P}'[|M'_n-X'_n|>\epsilon] = 0,\quad 1\leq n\leq n'.
  \]
  From the definition of~$M'$ and~$X'$, this is clear for $n<n'$, and for
  $n=n'$ it is obvious that $|M'_n-X'_n|\leq \epsilon$ on $\Omega\setminus U$.
  For an arbitrary element $\omega_{rj}^{(i)}$, we have
  \begin{align*}
    |M'_{n'}(\omega_{rj}^{(i)}) - X'_{n'}(\omega_{rj}^{(i)})|
      &= |M_{n'}(\omega_{rj}) - \hat{X}_{n'}(\hat{\omega}_i)| \\
      &= |\hat{M}_{n'}(\hat{\omega}_i) - \hat{X}_{n'}(\hat{\omega}_i)| \leq \epsilon.
  \end{align*}
  The last inequality follows from~\eqref{eq:XM}, as we may assume w.l.o.g.\ that
  $\hat{\mathbb P}$ puts mass on all elements of~$\hat \Omega$.
  
  Recall that~$U$ was defined as the preimage of~$z$. Repeating the procedure we just described
  for all values in the range of~$M_{n'}$ completes the induction step.
\end{proof}

For the formulation of the main result of this section,
recall the definition of~$\Phi_N$ in~\eqref{eq:Phi}. Theorem~\ref{thm:strassenW}
holds for $\epsilon=0$, too; then it is just a special case of Strassen's theorem
(recall Proposition~\ref{prop:consistent} and~\eqref{eq:Phi consistent}).

\begin{thm}[a variant of Strassen's theorem]\label{thm:strassenW}
   Let $\epsilon>0$ and $(\mu_n)_{n =1,\dots, n_0}$ be a sequence of measures
   in $\setM$ with finite support such that 
   \[
     I:= \bigcap_{1\leq n\leq n_0}[\mathbb{E}\mu_n-\epsilon, \mathbb{E}\mu_n+\epsilon] \neq \emptyset.
   \]
%   \begin{equation}\label{eq:suppmufinite}
%      \bigcup_{1\leq n\leq n_0} \mathrm{supp} (\mu_n)\quad \text{is finite.}
%   \end{equation}
   Then the following conditions are equivalent:
   \begin{itemize}
     \item[(i)] For some $m \in I$
   and for all $1\leq N<n_0$ and  $x_1, \dots, x_N \in \setR$, %\cup\{\pm \infty\}$, 
   we have
   \begin{equation*}
     \Phi_N(x_1,\dots,x_N;m,\epsilon)\leq0.
   \end{equation*}
     \item[(ii)] There is a filtered probability space $(\Omega,\mathcal F,(\mathcal{F}_n)_{1\leq n\leq n_0},
     \mathbb P)$ supporting two processes~$M$ and $X$ such that
     \begin{itemize}
        \item $M$ is a martingale w.r.t.\ $(\mathcal{F}_n)_{1\leq n\leq n_0}$
        \item $X$ is adapted to $(\mathcal{F}_n)_{1\leq n\leq n_0}$
        \item $M_n\sim \nu_n, \quad n=1,\dots,n_0$,
        \item $X_n\sim \mu_n, \quad n=1,\dots,n_0$,
        \item $\mathbb{P}[|M_n-X_n|>\epsilon]=0, \quad n=1,\dots,n_0$.
  \end{itemize}
   \end{itemize}
   \begin{proof}
     Suppose that (ii) holds. Since $\mathbb{P}[|M_n-X_n|>\epsilon]=0$, we have
     $W^\infty(\mu_n,\nu_n)\leq \epsilon$. As~$M$ is a martingale, $(\nu_n)$
     is a peacock, and so~(i) follows from (the easy implication of) Theorem~\ref{strassenwinf}.
     
     Now assume that~(i) holds. Then Theorem~\ref{strassenwinf} yields a peacock
     $(\nu_n)_{1\leq n\leq n_0}$ satisfying $W^\infty(\mu_n,\nu_n)\leq \epsilon$
     for $1\leq n\leq n_0$.
     %Going through the proof of Theorem~\ref{strassenwinf} one more time and 
     Using Corollary~\ref{cor:munufinite}, we see that the finiteness of the support of
     the~$\mu_n$ implies that we can choose $(\nu_n)_{1\leq n\leq n_0}$ with finite support, too.
     %such that $\bigcup_{1\leq n\leq n_0} \mathrm{supp} (\nu_n)$ is finite.
     %
     From Strassen's theorem we get a filtered probability space
  $(\Omega^*,\mathcal{F}^*,(\mathcal{F}^*_n),\mathbb{P}^*)$ with a martingale~$M^*$ satisfying
  $M_n^*\sim \nu_n$ for $1\leq n\leq n_0$. Moreover, as $W^\infty(\mu_n,\nu_n)\leq \epsilon$,
  there is a probability space $(\hat \Omega,\hat{\mathcal F},\hat{\mathbb P})$
  with two processes~$\hat M$ and~$\hat X$ satisfying $\hat{M}_n\sim \nu_n$,
  $\hat{X}_n\sim \mu_n$ for all~$n$ and
  \[
    \hat{\mathbb P}[|\hat{M}_n-\hat{X}_n|>\epsilon]=0, \quad n=1,\dots,n_0.
  \]
  (This is an easy consequence of Proposition~\ref{prop:DudleyStrassen} and the finiteness of the supports 
  of~$\mu_n$ and~$\nu_n$.)
  We may assume that both~$\Omega^*$ and~$\hat{\Omega}$ are finite. Indeed, we may clearly
  replace them by the finite sets
  \[
     \text{all intersections of sets from}\quad \{(M^*_n)^{-1}(z):  z \in
       \mathrm{supp}(\nu_n),\ 1\leq n\leq n_0 \} 
  \]
  respectively
   \begin{align*}
     \text{all intersections of sets from}&\quad \{\hat{M}_n^{-1}(z): z \in
       \mathrm{supp}(\nu_n),\ 1\leq n\leq n_0 \} \\
     \text{and}&\quad \{\hat{X}_n^{-1}(z): z \in
       \mathrm{supp}(\mu_n),\ 1\leq n\leq n_0 \}
  \end{align*}
  and update the sigma-algebras and the filtration of~$\Omega^*$ accordingly.
  The assertion then follows from Lemma~\ref{le:finite}.
  \end{proof}
\end{thm}
In future work, we intend to prove an appropriate version of Theorem~\ref{thm:strassenW}
(possibly featuring $d_0^{\mathrm{P}}$ or $d_p^{\mathrm{P}}$ instead of $W^\infty$)
for infinite sequences of general probability measures.
Also, a natural problem is to extend our peacock approximation results to other distances,
such as the $p$-Wasserstein distance $W^p$ ($p\geq 1$). Note that a related problem (involving the \emph{sum} of the $W^2$-distances
of all sequence elements) has been solved in~\cite{Ru85}.

\bibliographystyle{siam}
\bibliography{literatur}

\end{document}